\documentclass[10pt]{amsart}
\usepackage{eepic}
\usepackage{epsfig}
\usepackage{graphicx}
\usepackage{eufrak}
\usepackage{mathrsfs}
\usepackage[english]{babel}
\usepackage{color}

\usepackage{palatino, mathpazo}
\usepackage{amscd}      
\usepackage{amssymb}
\usepackage{amsmath}
\usepackage{dsfont}
\usepackage{xypic}      
\LaTeXdiagrams          
\usepackage[all,v2]{xy}
\xyoption{2cell} \UseAllTwocells \xyoption{frame} \CompileMatrices
\usepackage{latexsym}
\usepackage{epsfig}
\usepackage{amsfonts}
\usepackage{enumerate}

\allowdisplaybreaks[3]

\newtheorem{prop}{Proposition}[section]
\newtheorem{lem}[prop]{Lemma}

\newtheorem{cor}[prop]{Corollary}
\newtheorem{thm}[prop]{Theorem}
\newtheorem{rmk}[prop]{Remark}

\newtheorem{defn}[prop]{Definition}

\newtheorem{example}{Example}
            {\nolinebreak $\Box$ \end{trivlist}}

\newcommand{\noprint}[1]{}

\newcommand{\Ext}{\mbox{Ext}}

\newcommand{\Hom}{\mbox{Hom}}

\newcommand{\virt}{\mbox{\tiny virt}}

\newcommand{\zz}{{\mathbb Z}}

\newcommand{\aaa}{{\mathbb A}}

\renewcommand{\ll}{{\mathbb L}}

\newcommand{\cc}{{\mathbb C}}
\newcommand{\G}{{\mathbb G}}

\newcommand{\ff}{{\mathbb F}}

\newcommand{\sS}{{\mathcal S}}

\newcommand{\oO}{{\mathcal O}}

\newcommand{\Coh}{\mbox{Coh}}

\DeclareMathOperator{\loc}{loc}

\DeclareMathOperator{\id}{id}
\DeclareMathOperator{\Crit}{Crit}

\DeclareMathOperator{\DT}{DT}

\DeclareMathOperator{\Ob}{Ob}
\DeclareMathOperator{\eu}{Eu}
\DeclareMathOperator{\vd}{vd}

\newcommand{\bX}{{\mathbf X}}
\newcommand{\bt}{{\mathbf t}}

\newcommand{\Proj}{\mathop{\rm\bf Proj}\nolimits}

\newcommand{\red}{\mathop{\rm red}\nolimits}

\newcommand{\spec}{\mathop{\rm Spec}\nolimits}

\newcommand{\Sym}{\mathop{\rm Sym}\nolimits}
\newcommand{\proj}{\mathop{\rm Proj}\nolimits}

\numberwithin{equation}{subsection}

\def\Label{\label}

\title[Symmetric semi-perfect obstruction theory]{Symmetric semi-perfect obstruction theory revisited}
\author{Yunfeng Jiang}

\address{Department of Mathematics\\ University of Kansas\\ 405 Snow Hall 1460 Jayhawk Blvd\\Lawrence KS 66045 USA} 
\email{y.jiang@ku.edu}



\begin{document}
\sloppy \maketitle
\begin{abstract}
In this paper we survey some results on the symmetric semi-perfect obstruction theory on a Deligne-Mumford stack $X$ constructed by Chang-Li, and  Behrend's theorem equating the weighted Euler characteristic of  
$X$ and the virtual count of $X$ by symmetric semi-perfect obstruction theories.  
As an application, we prove that Joyce's d-critical scheme admits a symmetric semi-perfect obstruction theory, which can be applied to the  virtual Euler characteristics by Jiang-Thomas.   
\end{abstract}

\maketitle

\tableofcontents

\section{Introduction}

\subsection{Symmetric semi-perfect obstruction theory}

Perfect obstruction theories and virtual fundamental cycles are powerful tools to define counting invariants in modern enumerative geometry.  The construction was given by Li-Tian \cite{LT}, Behrend-Fantechi \cite{BF} in algebraic geometry.  The techniques apply to both Gromov-Witten theory and Donadson-Thomas theory.  
However there exist some  moduli problems such that the perfect obstruction theory only locally exists.  
In \cite{CL}, Chang-Li generlize and define semi-perfect obstruction theory which is locally given by a perfect obstruction theory.  The local virtual cycles glue to give a global virtual fundamental cycle.   Examples of schemes admitting semi-perfect obstruction theory include the moduli spaces of derived objects in the derived category of coherent sheaves on a Calabi-Yau threefold.

In \cite{Behrend}, Behrend introduced the ``symmetric obstruction theory" in order to prove that the Donaldson-Thomas invariants on a Calabi-Yau threefold are motivic invariants.   The symmetry on the obstruction theory  implies special property for the obstruction theory, see \S \ref{subsec_symmetric_SPOT}.   Behrend proved that the virtual count of the symmetric obstruction theory on a scheme or Deligne-Mumford stack $X$ is a weighted Euler characteristic of $X$ weighted by the ``Behrend function", see \S \ref{sec_Behrend_theorem} and more details on the Behrend function is included in \cite{Jiang2}.

In this paper we survey some results for symmetric semi-perfect obstruction theories, and  Behrend's theorem on equating the virtual count to the weighted Euler characteristic by the Behrend function.  

\subsection{Application to Donaldson-Thomas invariants}
Let $Y$ be a smooth Calabi-Yau threefold or a smooth threefold Calabi-Yau Deligne-Mumford stack. The Donaldson-Thomas invariants of $Y$ count stable coherent sheaves.  R. Thomas   \cite{Thomas} constructed a perfect obstruction theory $E^\bullet$ in the sense of Li-Tian \cite{LT}, and Behrend-Fantechi \cite{BF} on the moduli space $X$ of stable sheaves over $Y$.   If $X$ is proper, then the virtual dimension of $X$ is zero, and the integral
$\DT_Y=\int_{[X]^{\virt}}1$
is the Donaldson-Thomas invariant of $Y$. In \cite{Behrend} Behrend proved
 that the moduli scheme $X$ of stable sheaves on $Y$ admits a symmetric obstruction theory.  Let  
$$\nu_{X}: X\to\zz$$
be the Behrend function 
on $X$.   If $X$ is proper, then in \cite[Theorem 4.18]{Behrend} Behrend proved that 
$\DT_Y=\int_{[X]^{\virt}}1=\chi(X,\nu_{X})$, 
where $\chi(X,\nu_{X})$ is the weighted Euler characteristic weighted by the Behrend function.  Same result for a proper Deligne-Mumford stack $X$ with a symmetric perfect obstruction theory is conjectured by Behrend in \cite{Behrend}, and  is proved in \cite{Jiang3}.

\subsection{$D$-critical schemes}

In \cite{BBJ} and \cite{BBBJ}, Joyce etc studied the classical underlying scheme of a $(-1)$-shifted derived symplectic scheme defined in \cite{PTVV}.  
In \cite{Joyce}, Joyce introduced the notion of $d$-critical schemes or $d$-critical analytic spaces, where the underlying scheme of a $(-1)$-shifted derived symplectic scheme defined in \cite{PTVV} is a $d$-critical scheme. 
For instance, 
the moduli space $X$ of stable coherent sheaves or simple complexes on a Calabi-Yau threefold $Y$ can be lifted to a $(-1)$-shifted derived symplectic scheme, and is a $d$-critical scheme.  But the notion of $d$-critical schemes is more general than the underlying scheme of a $(-1)$-shifted derived symplectic scheme.  Kiem-Li \cite{KL1} introduced ``virtual critical manifolds", and proved that they are equivalent to $d$-critical analytic spaces. The notion of $d$-critical schemes or virtual critical manifolds are important to the categorification of Donaldson-Thomas invariants. 

We provide the detail proof that a $d$-critical scheme $X$ admits a symmetric semi-perfect obstruction theory. Kiem-Li \cite{KL1} have already proved that a virtual critical manifold admits a symmetric semi-perfect obstruction theory.  All results here in this section are due to Kiem-Li.

\subsection{Virtual signed Euler characteristics}

Motivated by the cotangent theory of Costello \cite{Costello} and Gromov-Witten theory with p-fields in \cite{CL1},  in  \cite{JT} R. Thomas and  the author defined four virtual signed Euler characteristics on the dual obstruction sheaf $N=\Ob^*_X$ of a perfect obstruction theory on a scheme $X$.   The abelian cone $N$ locally is given by the critical locus of a regular function on a higher dimensional  smooth scheme, and $N$ is a $d$-critical scheme of Joyce \cite{Joyce}.  In \cite{JT}, the scheme $X$ with a perfect obstruction theory $E_X^\bullet$ is assumed to the underlying scheme of a quasi-smooth derived scheme so that on the cone $N$ there exists a symmetric obstruction theory by pullback from the derived cotangent bundle of $(X, E_X^\bullet)$.   Since $N$ is a $d$-critical scheme, $N$ naturally admits a symmetric semi-perfect obstruction theory, and we still can define the four virtual signed Euler characteristics in \cite{JT}.  Note that in \cite{Kiem}, Kiem has already proved that on $N$ there is a semi-perfect obstruction theory.  We also talk about a natural compactification $\overline{N}$ of $N$ by taking the projective cone of $\Ob_X$, and show that $\overline{N}$ is also a $d$-critical scheme under some conditions. 
This may have applications for the Vafa-Witten invariants as developed in \cite{TT1}. 

In a different situation, if $X$ is a scheme or a Deligne-Mumford stack admitting a symmetric semi-perfect obstruction theory,  and furthermore there exists a $\G_m$-action on $X$ such that the semi-perfect obstruction theory is $\G_m$-equivariant, we show that the Behrend's weighted Euler characteristic of $X$ is the same as the Kiem-Li cosection localized invariants of the $\G_m$-action.  This generalizes the result in \cite{Jiang2} to symmetric semi-perfect obstruction theories. 

\subsection{Outline}

 The rough structure of the paper is as follows.  We review the basic materials about symmetric semi-perfect obstruction theories in  \S \ref{sec:SPOT}, where  \S \ref{subsec:POT} goes over the perfect obstruction theory and obstruction space;  and in \S \ref{subsec_symmetric_SPOT}  the basic notion of symmetric semi-perfect obstruction theory is reviewed.  In \S \ref{sec_Behrend_theorem} we prove  Theorem \ref{thm_Behrend_Euler_semiPOT} that equating the weighted Euler characteristic to the virtual count of a symmetric semi-perfect obstruction theory.  
We apply the result in the theorem to $d$-critical schemes in \S \ref{sec_d_critical_schemes};  where 
in \S \ref{subsec_d_critical_scheme} we review Joyce's $d$-critical schemes, in 
\S \ref{subsec_SPOT_d_critical_schemes} we provide detail proof for a $d$-critical scheme admitting a symmetric semi-perfect obstruction theory; and finally 
in \S \ref{sec_Virtual_Euler}   two special cases of virtual count are generalized to symmetric semi-perfect obstruction theories. 

\subsection*{Convention}

Throughout the paper we work over an  algebraically closed field $\kappa$ of character zero.  
We use $\G_m$ to represent the one parameter multiplication group of $\aaa_{\kappa}^1$.
For a scheme or a Deligne-Mumford stack $X$, $L_X:=\ll^{\geq -1}_{X}$ represents the truncated cotangent complex, where $\ll_X$ is the full cotangent complex.  
We use the notation $E^{-i} := (E_i)^*$ for dual vector bundles, and reserve $\vee$ for the derived dual of coherent sheaves and complexes.

\subsection*{Acknowledgments}

The author would like to thank Huai-Liang Chang, Felix Janda and Richard Thomas for valuable discussions on symmetric semi-perfect obstruction theories. 
Many thanks for the hospitality of Huai-Liang Chang when visiting  HKUST in the summer 2018 where part of the work was done. 
This work is partially supported by  NSF Grant DMS-1600997.

\section{Symmetric semi-perfect obstruction theory}\label{sec:SPOT}

In this section we review some basic materials about the symmetric semi-perfect obstruction theory in \cite{CL}, \cite{Behrend}.  

\subsection{Perfect obstruction theory and obstruction spaces}\label{subsec:POT}
 
\begin{defn}\label{defn:semi_POT}
Let $X_\alpha$ be a scheme.  A perfect obstruction theory of $X_\alpha$ consists of a derived category morphism:
$$\phi: E\to L_{X_\alpha}$$
such that 
\begin{enumerate}
\item $E\in D(\oO_{X_\alpha})$ is perfect, of perfect amplitude contained in the interval $[-1, 0]$,
\item $\phi$ induces an isomorphism on $h^0$ and an epimorphism on $h^{-1}$.
\end{enumerate}
The sheaf $\Ob_{\phi}:=h^1(E^{\vee})$ is called the ``obstruction sheaf". 
\end{defn}

From \cite{BF},  the intrinsic normal cone $\mathbf{c}_{X_{\alpha}}$ is defined as follows: whenever there is a closed immersion $X_\alpha\to M$ into a smooth scheme $M$, $\mathbf{c}_{X_{\alpha}}\cong [C_{X_{\alpha}/M}/T_{M}|_{X_\alpha}]$, where $C_{X_\alpha/M}$ is the normal cone and $T_M|_{X_\alpha}$ is the restriction of the tangent bundle $T_M$ to $X_\alpha$ which acts on the normal cone.  The morphism $\phi: E\to L_{X_\alpha}$ defines a closed immersion of cone stacks
$$\mathbf{c}_{X_{\alpha}}\subset N_{X_{\alpha}}:=h^1/h^0(L_{X_\alpha}^{\vee}).$$
Let $[\mathbf{c}_{X_{\alpha}}]\in Z_*(N_{X_\alpha})$ be the associated cycle and let 
$$h^1/h^0(\phi): N_{X_\alpha}\hookrightarrow h^1/h^0(E^{\vee})$$
be the morphism 
 induced by the truncated perfect obstruction theory. 

\subsubsection{Lifting problem}

We recall the classical lifting problem \cite[\S 4]{BF}, \cite[Definition 2.5]{CL}. 

\begin{defn}\label{defn:lifting}
Let $\iota: T\to \overline{T}$ be a closed subscheme with $\overline{T}$ local Artinian.  Let $I$ be the ideal of $T$ in 
$\overline{T}$, and let $\mathfrak{m}$ be the ideal sheaf of the closed point of $\overline{T}$.  We call $\iota$ a small extension if $I\cdot \mathfrak{m}=0$. Given a small extension $(T, \overline{T}, I, \mathfrak{m})$ fitting into the commutative diagram:
\begin{equation}\label{diagram_lifting}
\xymatrix{
T\ar[r]^{g}\ar[d]& X_{\alpha}\ar[d]\\
\overline{T}\ar@{-->}[ur]\ar[r]& \spec(\kappa)
}
\end{equation}
so that $Im(g)$ contains a closed point $p\in X_\alpha$.  

Finding a morphism $\overline{g}: \overline{T}\to X_\alpha$ making the diagram (\ref{diagram_lifting}) commute is called the ``infinitesimal lifting problem of $X_\alpha$ at $p$". 
\end{defn}

Standard obstruction theory tells us, see \cite[Chapter 3, Theorem 2.1.7]{Ill}, \cite[Lemma 2.6]{CL}, that for an infinitesimal lifting problem of $X_\alpha$ at $p$ as in (\ref{diagram_lifting}), there exists a canonical element 
$$\omega(g, T, \overline{T})\in \Ext^1(g^\star L_{X_\alpha}, I)=T_{p, X_{\alpha}}\otimes_{\kappa}I$$
where $T_{p, X_{\alpha}}=h^1(L_{X_\alpha}^{\vee}|_{p})$ is the intrinsic obstruction space to deforming $p\in X_\alpha$. 
The vanishing of $\omega(g, T, \overline{T})$ is necessary and sufficient for the lifting problem to be solvable, in which case the collection of the solutions form a torsor under 
$\Ext^0(g^\star L_{X_\alpha}, I)=\Hom(g^\star\Omega_{X_\alpha}, I)$. 
\begin{rmk}
Recall that if $\phi: E\to L_{X_\alpha}$ is a perfect obstruction theory, then the obstruction space (of the obstruction theory $\phi$) to deforming $p\in X_\alpha$ is defined to be $\Ob(\phi, p)=h^1(E^\vee|_{p})$. 
\end{rmk}

\begin{defn}\label{defn:obstruction:class}
Let $\phi: E\to L_{X_\alpha}$ be  a perfect obstruction theory. For an infinitesimal lifting problem (\ref{diagram_lifting}), the image 
$$\Ob(\phi, g, T, \overline{T}):=h^1(\phi^\vee)(\omega(g, T, \overline{T}))\in \Ext^1(g^\star E, I)=\Ob(\phi,p)\otimes_{\kappa}I$$
is called ``the obstruction class" of $\phi$ to the lifting problem (\ref{diagram_lifting}). 
\end{defn}

From \cite[Theorem 4.5]{BF}, the obstruction class $\Ob(\phi, g, T, \overline{T})=0$ if and only if the lifting problem (\ref{diagram_lifting}) is solvable. 

Let $\phi: E\to L_{X_\alpha}$ and $\phi^\prime: E^\prime\to L_{X_\alpha}$ be two perfect obstruction theories. 

\begin{defn}\label{defn_nu_equivalence}
We call $\phi$ is $\nu$-equivalent to $\phi^\prime$ if there exists an isomorphism of sheaves 
\begin{equation}\label{eqn_nu}
\psi: h^1(E^\vee)\stackrel{\cong}{\longrightarrow} h^1(E^{\prime \vee})
\end{equation}
such that for every closed point $p\in X_\alpha$, and for any ``infinitesimal lifting problem" of $X_\alpha$ at $p$ as in the diagram (\ref{diagram_lifting}), we have
$$\psi|_{p}(\Ob(\phi, g, T,\overline{T}))=\Ob(\phi^\prime, g, T, \overline{T})\in \Ob(\phi^\prime ,p)\otimes_{\kappa}I.$$
\end{defn}

As in \cite[Formula (2.4)]{CL}, let 
$$\eta_{\phi}: N_{X_\alpha}\to h^1/h^0(E^\vee)\to h^1(E^\vee)=\Ob_{\phi}$$
be the composition morphism, where the first arrow is an embedding of the intrinsic normal sheaf to the bundle stack for the obstruction theory $\phi: E\to L_{X_\alpha}$.  Let 
$$[\mathbf{c}_{\phi}]=\eta_{\phi \star}[\mathbf{c}_{X_\alpha}]$$
be the image of the intrinsic normal cone $\mathbf{c}_{X_\alpha}$.

\begin{prop}(\cite[Proposition 2.10]{CL})
If $\phi$ and $\phi^\prime$ are two $\nu$-equivalent perfect obstruction theories, and let 
$$\eta_\phi: N_{X_\alpha}\to h^1(E^\vee)=\Ob_{\phi}$$
be the morphism above, then for any integral cycle $A\subset N_{X_\alpha}$, we have 
$$\psi_{\star}(\eta_{\phi \star}[A])=\eta_{\phi^\prime \star}[A]\in Z_*(h^1(E^{\prime \vee}))=Z_*(\Ob_{\phi^\prime}).$$
\end{prop}

Let $\mathbf{c}_{X_\alpha}$ be the intrinsic normal cone of $X_\alpha$. Then 
\begin{align*}
\psi_{\star}(\eta_{\phi \star}[\mathbf{c}_{X_\alpha}])
&=\psi_{\star}([\mathbf{c}_{\phi}])\\
&=\eta_{\phi^\prime \star}[\mathbf{c}_{X_\alpha}]\\
&=[\mathbf{c}_{\phi^\prime}].
\end{align*}

\subsection{Symmetric semi-perfect obstruction theory}\label{subsec_symmetric_SPOT}

Let $X$ be a Deligne-Mumford stack of locally of finite type.  For any two {\'etale morphisms 
$X_\alpha, X_\beta\to X$ we let $X_{\alpha\beta}=X_\alpha\times _{X}X_{\beta}$, and for any 
object $F\in D^b(X_\alpha)$ in the derived category, $F|_{X_{\alpha\beta}}$ is  the pullback of $F$ under the projection 
$X_{\alpha\beta}\to X_\alpha$. 

\begin{defn}(\cite[Definition 3.1]{CL})\label{defn_semiPOT}
A semi-perfect obstruction theory of $X$ consists of an \'etale covering $\{X_{\alpha}\}_{\alpha\in\Lambda}$ of 
$X$ by affine schemes, and truncated perfect obstructin theories
$$\phi_{\alpha}: E_{\alpha}\to L_{X_\alpha}, \alpha\in\Lambda$$
such that 
\begin{enumerate}
\item for any pair $\alpha, \beta\in \Lambda$ there exists an isomorphism 
\begin{equation}\label{isomorphism_psi_alphabeta}
\psi_{\alpha\beta}: h^1(E_{\alpha}^\vee)|_{X_{\alpha\beta}}\stackrel{\cong}{\longrightarrow}
h^1(E_{\beta}^\vee)|_{X_{\alpha\beta}}
\end{equation}
such that the collections $(h^1(E_{\alpha}^\vee), \psi_{\alpha\beta})$ forms a descent data of sheaves.
\item for any pair $\alpha, \beta\in \Lambda$, the obstruction theories 
$\phi_{\alpha}|_{X_{\alpha\beta}}$ and $\phi_{\beta}|_{X_{\alpha\beta}}$ are $\nu$-equivalent via $\psi_{\alpha\beta}$. 
\end{enumerate}
Of course, a perfect obstruction theory is a semi-perfect obstruction theory. 
\end{defn}

Let $\phi=\{\phi_\alpha, X_\alpha, E_\alpha, \psi_{\alpha\beta}\}_{\alpha\in\Lambda}$ be a semi-perfect obstruction theory for $X$.  We denote by $\Ob_{\phi}$ the resulting obstruction sheaf of the semi-perfect obstruction theory by the gluing of $\Ob_{\phi}$ for $\alpha\in \Lambda$. 

Let $N_{X}=h^1/h^0(L_X^\vee)$ be the intrinsic normal sheaf of $X$, and we can think of this sheaf as the gluing of 
$N_{X}|_{X_\alpha}=N_{X_\alpha}=h^1/h^0(L_{X_\alpha}^\vee)$ for $\alpha\in \Lambda$. Then there exists a group homomorphism
$$\eta_{\star}: Z_*(N_X)\to Z_*(\Ob_\phi)$$
by patching the collection:
$$\eta_{\phi_{\alpha} \star}: Z_*(N_{X_\alpha})\to Z_*(\Ob_{\phi_\alpha}).$$
\cite[Lemma 3.3]{CL} proves that for any integral Artin stack $[A]\in Z_*(N_{X})$, the collection 
$$[A_\alpha]:=\eta_{\phi_{\alpha} \star}[A\times_{X}X_{\alpha}]\in Z_*(\Ob_{\phi}|_{X_\alpha})$$
satisfy the descent condition:
\begin{equation}
A_\alpha\times_{X_\alpha}X_{\alpha\beta}=A_\beta\times_{X_\beta}X_{\alpha\beta}\subset 
\Ob_{\phi}|_{X_{\alpha\beta}}
\end{equation}
so that it forms an integral cycle in $Z_*(\Ob_{\phi})$.  Let $\eta_{\phi \star}: Z_*(N_X)\to Z_*(\Ob_{\phi})$ be the homomorphism by linear extensions. 

Let $\mathbf{c}_{X}$ be the intrinsic normal cone of $X$ such that 
\'etale locally on $X_\alpha\to X$, there exists a closed immersion
$$X_\alpha\hookrightarrow M$$
into a smooth scheme $M$, $\mathbf{c}_{X}|_{X_\alpha}=[C_{X_\alpha/M}/T_M|_{X_\alpha}]$. 
Then $[\mathbf{c}_{X}]\in Z_*(N_X)$ is a cycle, and we define 
$$[cv_X]=\eta_{\phi\star}[\mathbf{c}_{X}]\in Z_*(\Ob_\phi).$$
Let $s: X\to \Ob_{\phi}$ be the zero section, then 
\cite[Proposition 3.4]{CL} constructed the Gysin map 
$$s^{!}:  Z_*(\Ob_\phi)\to Z_*(X)$$
such that 
$$[X, \phi]^{\virt}:=s^{!}([cv_{X}])\in A_*(X)$$
is the virtual fundamental cycle associated with the semi-perfect obstruction theory 
$\phi$. 

\subsubsection{Obstruction cone}

We generalize the obstruction cone to semi-perfect obstruction theory. 

Recall from \cite[\S 2.1]{Behrend}, a local resolution of $\phi=\{\phi_\alpha, X_\alpha, E_\alpha, \psi_{\alpha\beta}\}_{\alpha\in\Lambda}$ is a derived category homomorphism 
$$F\to E_{\alpha}^\vee[1]|_{U}$$
over some \'etale open chart $U$ of $X_\alpha$, where $F$ is a vector bundle over $U$ and the homomorphism $F\to E_{\alpha}^\vee[1]|_{U}$ satisfies the condition that its cone is a locally free sheaf over $U$ concentrated in degree $-1$.  Or as a local presentation  
$F\to h^1/h^0(E_{\alpha}^\vee)|_{U}$ of the bundle stack over $U$ of $X_\alpha$. 

For every local resolution $F\to h^1/h^0(E_{\alpha}^\vee)|_{U}$  there exists an associated cone $C\subset F$, the obstruction cone, defined via the Cartesian diagram:
\[
\xymatrix{
C\ar[r]\ar[d] &F\ar[d]\\
\mathbf{c}_{X}|_{U}\ar[r]& h^1/h^0(E_{X_\alpha}^\vee)|_{U}
}
\]
where $\mathbf{c}_{X}|_{U}\cong \mathbf{c}_{X_\alpha}|_{U}$.  The local resolution 
$F\to h^1/h^0(E_{\alpha}^\vee)|_{U}$ induces a canonical epimorphism 
$$F\to \Ob_{\phi_{\alpha}}|_{U}=\Ob_{\phi}|_{U}$$
of coherent sheaves. 

\begin{prop}(\cite[Proposition 2.2]{Behrend})\label{prop_obstruction_cone}
Let $\Omega$ be a vector bundle over $X$, and $\Omega\to \Ob_{\phi}$ an epimorphism of coherent sheaves. Then there exists a unique closed subcone $C\subset \Omega$ such that for every local resolution 
$F\to E_{\alpha}^\vee[1]|_{U}$, with obstruction cone $C^\prime\subset F$, and every lift $\gamma$
\[
\xymatrix{
&F\ar[d]\ar[dl]^{\gamma}\\
\Omega|_{U}\ar[r]& \Ob_{\phi}|_{U}
}
\]
we have $C|_{U}=\gamma^{-1}(C^\prime)$, in the scheme-theoretic sense. 
\end{prop}
\begin{proof}
Still since for $\alpha\in \Lambda$,  $\phi_{\alpha}: E_\alpha\to L_{X_\alpha}$ is a perfect obstruction theory. 
\'Etale locally around $X_\alpha$ (hence around $X$), the presentation $F$ and $\gamma$ always exists. Hence the uniqueness of $C$ is true.   We only need to prove its existence. 

Let $X_{\alpha\beta}\to X_\alpha$ be the embedding of \'etale chart. Then from \cite[\S 3]{BF}, the intrinsic normal cone 
$$\mathbf{c}_{X_{\alpha\beta}}=\mathbf{c}_{X}|_{X_{\alpha\beta}}\hookrightarrow \mathbf{c}_{X}|_{X_{\alpha}}=\mathbf{c}_{X_{\alpha}}$$
is a closed subcone stack via the embedding of intrinsic normal sheaves $N_{X_{\alpha\beta}}\to N_{X_\alpha}$. 
On the other hand, the vector bundle stack 
$$h^1/h^0(E_\alpha^\vee)|_{X_{\alpha\beta}}\to h^1/h^0(E_\alpha^\vee)$$
is an embedding as closed bundle stacks, and 
\[
\xymatrix{
\mathbf{c}_{X_\alpha}|_{X_{\alpha\beta}}\ar@{^{(}->}[r]\ar@{^{(}->}[d]& N_{X_\alpha}|_{X_{\alpha\beta}}\ar@{^{(}->}[r]\ar@{^{(}->}[d]& h^1/h^0(E_\alpha^\vee)|_{X_{\alpha\beta}}\ar@{^{(}->}[d]\\
\mathbf{c}_{X_\alpha}\ar@{^{(}->}[r]& N_{X_\alpha}\ar@{^{(}->}[r]& h^1/h^0(E_\alpha^\vee)|_{X_{\alpha\beta}}
}
\]
as embedding diagrams. 

From the definition of semi-perfect obstruction theory, for any $\alpha, \beta$,  
$$\psi_{\alpha\beta}: h^1/h^0(E_\alpha^\vee)|_{X_{\alpha\beta}}\stackrel{\cong}{\longrightarrow}
h^1/h^0(E_\beta^\vee)|_{X_{\alpha\beta}},$$
hence the bundle stacks glue to give a stack $h^1/h^0$ on $X$ such that for each \'etale morphism 
$X_\alpha\to X$, $h^1/h^0|_{X_\alpha}\cong h^1/h^0(E_\alpha^\vee)$. 

Still let $cv$ be the coarse moduli sheaf of the intrinsic normal cone $\mathbf{c}_X$, and we get a Cartesian diagram on $X$:
\[
\xymatrix{
\mathbf{c}_X\ar@{^{(}->}[r]\ar[d]& h^1/h^0\ar[d]\\
cv\ar@{^{(}->}[r]&\Ob_{\phi}
}
\]
as stacks (in the big \'etale site of $X$).  Then the cone $C\hookrightarrow \Omega$ is constructed by the fibre product of sheaves on the big \'etale site of $X$:
\begin{equation}\label{eqn_cone_construction}
\xymatrix{
C\ar@{^{(}->}[r]\ar[d]& \Omega\ar[d]\\
cv\ar@{^{(}->}[r]&\Ob_{\phi}
}
\end{equation}
which is Cartesian. So any $\Omega\to F$ gives a diagram 
\[
\xymatrix{
C\ar@{^{(}->}[r]\ar[d]& \Omega\ar[d]\\
C^\prime\ar@{^{(}->}[r]&F
}
\]
which is Cartesian because (\ref{eqn_cone_construction}) and the Cartesian diagram:
\[
\xymatrix{
C^\prime\ar@{^{(}->}[r]\ar[d]& F\ar[d]\\
cv\ar@{^{(}->}[r]&\Ob_{\phi}
}
\]
by assuming $F$ is a global resolution. 
\end{proof}

\begin{defn}\label{defn_symmetric_semiPOT}
A semi-perfect obstruction theory $\phi=\{\phi_\alpha, X_\alpha, E_\alpha, \psi_{\alpha\beta}\}_{\alpha\in\Lambda}$ for $X$ is $symmetric$ if for any 
$\alpha\in\Lambda$, $E_\alpha$ is endowed with a non-degenerate symmetric bilinear form 
$\theta_{\alpha}: E_\alpha\stackrel{\cong}{\longrightarrow} E_{\alpha}^\vee[1]$.
\end{defn}

\begin{prop}
The obstruction sheaf $\Ob_\phi$ for a symmetric semi-perfect obstruction theory 
$\phi$ is the cotangent sheaf $\Omega_X$. 
\end{prop}
\begin{proof}
Locally the obstruction sheaf 
$\Ob_{\phi_\alpha}=h^1(E_\alpha^\vee)=\Omega_{X_\alpha}$ is the cotangent sheaf of $X_\alpha$, therefore $\Omega_{X_\alpha}$ glue to give the cotangent sheaf $\Omega_{X}$. 
\end{proof}

\subsubsection{Almost closed one form}

Recall from \S 3.4 of \cite{Behrend}, any symmetric obstruction theory is locally given by an almost closed $1$-form.  This is still true for symmetric semi-perfect obstruction theory.  
Let $\phi=\{\phi_\alpha, X_\alpha, E_\alpha, \psi_{\alpha\beta}\}_{\alpha\in\Lambda}$ be a  symmetric semi-perfect obstruction theory for $X$.   The symmetric obstruction theory $E_\alpha$ on any \'etale local chart 
$X_\alpha\to X$ gives rise to the following:
for each point $p\in X_\alpha$ in an \'etale neighborhood, there exists an immersion 
$$X_\alpha\hookrightarrow M$$
into a smooth scheme $M$, such that there exists an almost closed one form 
$\omega\in \Omega_M^1$ and an isometry $E_\alpha\to H(\omega)$
such that 
\[
\xymatrix{
E_\alpha\ar[rr]\ar[dr] &&H(\omega)\ar[dl]\\
&L_{X_\alpha}&
}
\]
where $H(\omega)=\Big[T_M|_{X_\alpha}\stackrel{\triangledown\omega}{\longrightarrow}\Omega_{M}|_{X_\alpha}\Big]$, $\triangledown\omega=d\cdot \omega^\vee$ coming from the following diagram:
\[
\xymatrix{
T_M|_{X_\alpha}\ar[r]^{\triangledown\omega}\ar[d]^{\omega^\vee}& \Omega_{M}|_{X_\alpha}\ar[d]\\
I_{X_\alpha}/I_{X_\alpha}^2\ar[r]^{d}& \Omega_{M}|_{X_\alpha}.
}
\]
The almost closed one form $\omega\in \Omega^1_{M}$ means that 
$d\omega\in I_{X_\alpha}\cdot \Omega_M^2$, where $I_{X_\alpha}$ is the ideal sheaf of the zero locus of 
$\omega$ ($X_\alpha$ is the zero locus of $\omega$). 

\subsubsection{Virtual fundamental cycle}

Let $\phi=\{\phi_\alpha, X_\alpha, E_\alpha, \psi_{\alpha\beta}\}_{\alpha\in\Lambda}$ be a  symmetric semi-perfect obstruction theory for $X$.  Recall that we have
the morphism 
$$\eta_{\phi \star}: Z_*(N_X)\to Z_*(\Omega_X)$$
constructed before. Let $\mathbf{c}_X$ be the intrinsic normal cone and 
$cv=\eta_{\phi\star}[\mathbf{c}_X]$ is exactly the coarse moduli sheaf of the intrinsic normal cone taken as a cycle in $Z_*(\Omega_{X})$. 

Let 
$$[X,\phi]^{\virt}=s^{!}_{\Omega_X}([cv])$$
be the virtual fundamental cycle by applying the Gysin map 
$$s^{!}_{\Omega_X}: Z_*(\Omega_X)\to A_*(X)$$ 
given by 
$$[cv]\mapsto [X,\phi]^{\virt}$$
as in \cite[Proposition 3.4]{CL}. 

We give an alternative construction due to \cite{Behrend}.  Note that since $\mathbf{c}_X$ and $h^1/h^0(E_\alpha^\vee)$ are all Artin stacks, one can use the intersection theory of Artin stacks in \cite{Kretch} to directly applying the Gysin map on the Chow group 
$A_*(h^1/h^0)$ on $X$. Let 
\[
\xymatrix{
C\ar@{^{(}->}[r]\ar[d]& \Omega\ar[d]\\
cv\ar@{^{(}->}[r]&\Ob_{\phi}=\Omega_X
}
\]
be the Cartesian diagram in (\ref{eqn_cone_construction}), where $\Omega$ is a vector bundle and $C$ is the obstruction cone in $\Omega$. 
\begin{prop}\label{prop_virt_obstruction_cone}
$$[X,\phi]^{\virt}=s_{\Omega}^{!}[C]\in A_*(X)$$
where $s^{!}_{\Omega}$ is the Gysin map of the vector bundle $\Omega \to X$. 
\end{prop}

\section{Behrend's theorem}\label{sec_Behrend_theorem}

In this section we survey the Behrend's theorem equating the virtual count of a symmetric semi-perfect obstruction theory $\phi$ to the weighted Euler characteristic of $X$.

\subsection{The Behrend function}

Let $X$ be a Deligne-Mumford stack. In \cite[\S 2]{Behrend} Behrend introduced on $X$ an integer valued constructible function 
$$\nu_X: X\to \zz,$$
which is called the ``Behrend function".  We briefly recall its construction.  More detail construction can be found in \cite{Behrend}, and \cite{Jiang2}. 

There exists a unique integral cycle $\mathfrak{c}_X$ on $X$ such that for any \'etale chart $U\to X$ and $U\to M$ an embedding into a smooth scheme $M$, 
$$\mathfrak{c}_X|_{U}=\mathfrak{c}_{U/M}$$
and
$$\mathfrak{c}_{U/M}=\sum_{C^\prime}(-1)^{\dim \pi(C^\prime)}\mbox{mult}(C^\prime)[\pi(C^\prime)]$$
where 
$$\pi: C_{U/M}\to U$$
 is the projection from the normal cone $C_{U/M}$ to $U$;  $C^\prime$ are all the irreducible components of the normal cone $C$; $\pi(C^\prime)$ are the irreducible closed subset (prime cycle) on $U$ by the image of $\pi$;  and 
$\mbox{mult}(C^\prime)$ is the multiplicity of $C^\prime$ at the generic point.

\begin{defn}\label{defn_Behrend_function}
The Behrend function $\nu_X: X\to\zz$ is defined as
$$\nu_X:=\mbox{Eu}(\mathfrak{c}_X)$$
where $\mbox{Eu}(-)$ is the local Euler obstruction of MacPherson  \cite{MacPherson} on integral cycles on $X$.\end{defn}

\begin{rmk}
\begin{enumerate}
\item
If $X$ is smooth, $\mathfrak{c}_X=(-1)^{\dim X}[X]$. In general it is a integral cycle in $Z_*(X)$. 
\item More motivation of the local Euler obstruction can be found in \cite{MacPherson}, see also \cite[\S 2]{Behrend}, \cite{Jiang2}.
\item If $X=\Crit(f)$ is the critical locus of a holomorphic function $f: M\to\cc$, then 
$$\nu_X(P)=(-1)^{\dim M}(1-\chi(\ff_P))$$
where $\ff_P$ is the Milnor fiber of the function $f$ at $P$. 
\end{enumerate}
\end{rmk}

\begin{defn}\label{defn_weighted_Euler}
The weighted Euler characteristic of $X$ by the Behrend function $\nu_X$ is defined as:
$$\chi(X,\nu_X)=\chi(X, \mbox{Eu}(\mathfrak{c}_X))=\sum_i i\cdot \chi(\nu_X^{-1}(i)).$$
\end{defn}

We recall the Aluffi class. First the Chern-Mather class is a group homomorphism
$$c^M: Z_*(X)\to A_*(X)$$
by linear extension for any prime cycle $V$ of degree $p$ on $X$, 
\begin{equation}\label{eqn_Chern_Mather}
c^M(V)=\mu_{\star}(c(T_{V})\cap [\widetilde V])
\end{equation}
where $\mu: \widetilde V\to V$ is the Nash blow-up, $T_{V}$ is the Nash tangent bundle on $\widetilde V$. 
Let $c^M_0(V)$ be the degree zero part of $c^M(V)$.  Behrend \cite[Definition 1.1]{Behrend} defines the Aluffi class as:
$$\alpha_X:=c^M(\mathfrak{c}_X)\in A_*(X).$$
If $X$ is smooth, then $\alpha_X=(-1)^{\dim X}c(T_X)\cap [X]=c(\Omega_X)\cap [X]$. 

\begin{thm}(\cite[Proposition 1.12]{Behrend}, \cite[Theorem 1.1]{Jiang3})\label{thm_MacPherson_index}
Let $X$ be a proper Deligne-Mumford stack. Then 
$$\int_{X}\alpha_X=\chi(X, \nu_X).$$
\end{thm}

\subsection{Behrend' theorem for symmetric semi-perfect obstruction theory}\label{subsec_Behrend_theorem}

\subsubsection{Lagrangian Cone} Let $M$ be a smooth scheme or a smooth Deligne-Mumford stack.  The cotangent bundle $\Omega_M$ has the tautological one form $\alpha\in \Omega_M$.  Any local \'etale coordinates 
$\{x_1, \cdots, x_n\}$ onn $M$ induces the canonical coordinate system 
$$\{x_1,\cdots, x_n, p_1, \cdots, p_n\}$$
on $\Omega_M$.  In such coordinate system $\alpha=\sum_i p_idx_i$. The differential $d\alpha=\theta$ defines the tautological symplectic form on $\Omega_M$.  

Recall that an irreducible closed subset $C\subset \Omega_M$ is conic and Lagrangian if and only if $\dim C=\dim M$ and $\alpha$ vanishes when restricted to the generic point of $C$. From \cite[Lemma 4.2]{Behrend}, if $V\subset M$ is an irreducible closed subset, the closure in $\Omega_M$ of the conormal bundle to any smooth dense open subset of $V$ is conic Lagrangian,  This way describes all conic Lagrangians. 

\begin{defn}\label{defn_conic_Lagrangian}
A closed subset of $\Omega_M$ is conic and Lagrangian if every one of its irreducible components is conic and Lagrangian.  An algebraic cycle in $\Omega_M$ is conic Lagrangian if its support is conic Lagrangian. 
\end{defn}

\subsubsection{Behrend's theorem}

Now since $X$ is a quasi-projective separated Deligne-Mumford stack, and let $X\hookrightarrow M$ be the closed immersion into a smooth projective Deligne-Mumford stack $M$ with projective coarse moduli space.  Moreover, $X$ admits a symmetric semi-perfect obstruction theory 
$\phi=\{\phi_\alpha, X_\alpha, E_\alpha, \psi_{\alpha\beta}\}_{\alpha\in\Lambda}$.  Let $C\subset \Omega_M$ be the obstruction cone constructed in Proposition \ref{prop_obstruction_cone} by 
\[
\xymatrix{
C\ar@{^{(}->}[r]\ar[d]& \Omega_M\ar[d]\\
cv\ar@{^{(}->}[r]& \Ob_{\phi}=\Omega_X
}
\]

\begin{thm}\label{thm_Lagrangian_cone}
The obstruction cone $C\subset \Omega_M$ is Lagrangian. 
\end{thm}
\begin{proof}
The problem is local, hence we may assume that $X$ is cut out by an almost closed one form 
$\omega\in\Omega^1_{M}$. Then the result is from \cite[Theorem 4.9]{Behrend}. 
\end{proof}

Recall that in \cite[\S 4]{Behrend}, let $\mathcal{L}_X(\Omega_M)$ is the subgroup of $\mathcal{Z}_n(\Omega_M)$ generated by the conic Lagrangian prime cycles supported on $X$. 
Let $V$ be a prime cycle of $M$, one has $N^*_{V/M}$  the closure of the conormal bundle of smooth part of $V$ inside $M$. 
So in 
 (Section 4.1, \cite{Behrend})   the following isomorphism of groups is defined:
\begin{equation}\label{cycle_Lagrangian}
L:  \mathcal{Z}_{*}(X) \to \mathfrak{L}_{X}(\Omega_{M})
\end{equation}
which is given by
$$V\mapsto (-1)^{\dim(V)}N^*_{V/M}.$$

Conversely there is an isomorphism:
\begin{equation}\label{Lagrangian_cycle}
\pi:  \mathfrak{L}_{X}(\Omega_{M})\to \mathcal{Z}_{*}(X) 
\end{equation}
which is given by
$$W\mapsto (-1)^{\dim(\pi(W))}\pi(W),$$
where $\pi: W\to X$ is the projection.  The homomorphisms $L$ and $\pi$ are inverse of each other.

Fix an embedding $X\to M$ of the DM stack $X$ into a smooth DM stack $M$. We have the following diagram due to Behrend in 
Diagram (2) of \cite{Behrend}.
\begin{equation}\label{Key_Diagram}
\xymatrix{
\mathcal{Z}_*(X)\ar[r]^{\eu}_{\cong}\ar[dr]_{c^{M}_0}&F(X)\ar[r]^{\mbox{Ch}}_{\cong}\ar[d]^{c_0^{SM}}&\mathfrak{L}_{X}(\Omega_{M})\ar[dl]^{s^{!}_{\Omega_M}(-)}\\
&A_0(X)&
}
\end{equation}
where $\mathcal{Z}_*(X)$ is the group of integral cycles of $X$, $F(X)$ is the group of constructible functions on $X$.  
The maps $c^{M}_0$, $c_0^{SM}$ and $I(\cdot, [M])$ are degree zero Chern-Mather class, degree zero Chern-Schwartz-Mather class and the Lagrangian intersection with zero section of $\Omega_{M}$, respectively.  Note that in \cite{Behrend}, the notation of Lagrangian intersection with zero section is denoted by $0^{!}_{\Omega_{M}}(\cdot)$. 

We briefly explain the horizontal morphisms in the diagram.  The first map is the local Euler obstruction $\eu$ and it gives an isomorphism from $\mathcal{Z}_*(X)$ to $F(X)$. 
Then the morphism $\mbox{Ch}$ is defined by the isomorphism $\eu$ and the morphism $L$ defined above. 

\begin{cor}(\cite[Corollary 4.15]{Behrend})\label{cor_cone_Ch}
We have 
$$[C]=L(\mathfrak{c}_X)=\mbox{Ch}(\nu_X). $$
\end{cor}

\begin{thm}\label{thm_Behrend_Euler_semiPOT}
Let $X$ be a proper Deligne-Mumford stack with a symmetric semi-perfect obstruction theory 
$\phi=\{\phi_\alpha, X_\alpha, E_\alpha, \psi_{\alpha\beta}\}_{\alpha\in \Lambda}$. Then 
$$\int_{[X,\phi]^{\virt}}1=\chi(X, \nu_X).$$
\end{thm}
\begin{proof}
From Proposition \ref{prop_virt_obstruction_cone}, and the fact that the virtual dimension is zero, 
$$\int_{[X,\phi]^{\virt}}1=\# s^{!}_{\Omega_M}([C]).$$
While from the diagram (\ref{Key_Diagram}) and Corollary \ref{cor_cone_Ch}, $\mbox{Ch}(\nu_X)=[C]$, so 
\begin{align*}
\# s^{!}_{\Omega_M}([C])&=\int_{X}c_0^{SM}(\nu_X)\\
&=\int_{X}c_0^{M}(\mathfrak{c}_X)\\
&=\chi(X,\nu_X)
\end{align*}
by Theorem \ref{thm_MacPherson_index}. 
\end{proof}
\begin{rmk}
By \cite[Proposition 3.2]{CL}, the invariant $\int_{[X,\phi]^{\virt}}1$ is deformation invariant. 
\end{rmk}

\section{$D$-critical schemes}\label{sec_d_critical_schemes}

\subsection{Introduction}

Motivated by Donaldson-Thomas theory and the $(-1)$-shifted symplectic derived schemes, Joyce introduced the notion of $d$-critical schemes or $d$-critical analytic spaces.  A parallel notion introduced by Kiem and Li \cite{KL1} is called virtual critical manifolds, and Kiem-Li proves that the $d$-critical analytic spaces and virtual critical manifolds are equivalent.  Classical examples include the moduli space of stable simple complexes on Calabi-Yau threefolds, and these moduli spaces are the underlying schemes of the lifted $(-1)$-shifted symplectic derived schemes of \cite{PTVV}.  These moduli schemes admits a natural symmetric obstruction theory. 

From Joyce \cite{Joyce},  it is not known if a $d$-critical scheme admits a symmetric obstruction theory, although locally it does.  In this section we prove that Joyce's $d$-critical scheme admits a symmetric semi-perfect obstruction theory.  Since a $d$-critical analytic space is equivalent to a virtual critical manifold in \cite{KL1}, and Kiem-Li already proved that virtual critical manifolds admits a symmetric semi-perfect obstruction theory, hence in the algebraic sense a $d$-critical scheme admits a symmetric semi-perfect obstruction theory.  We provide the detail proof for $d$-critical schemes and all credits belong to Kiem-Li in \cite{KL1}. 

As application we show that  the dual obstruction sheaf as in \cite{JT} is a $d$-critical scheme. We also show that its compactification is also a $d$-critical scheme under some restrictions. Therefore the dual obstruction sheaf admits naturally a  symmetric semi-perfect obstruction theory.  

Let $X$ be a scheme which admits a  symmetric semi-perfect obstruction theory, and furthermore assume that there exists a $\G_m$ action on $X$ which makes the symmetric semi-perfect obstruction theory $\G_m$-equivariant.  The $\G_m$-action naturally gives rise to a cosection and we show that the Kiem-Li localized invariant is equal to the Behrend's weighted Euler characteristic by the Behrend function.  This generalizes the result in \cite[Theorem 5.20]{Jiang2} to symmetric semi-perfect obstruction theory.

\subsection{Joyce's $d$-critical schemes}\Label{subsec_d_critical_scheme}

The algebraic $d$-critical scheme is the classical model for the $(-1)$-shifted symplectic derived scheme as developed by PTVV in \cite{PTVV}. In the same paper \cite{PTVV}, PTVV prove that the moduli space of stable coherent sheaves or simple complexes over Calabi-Yau threefolds admits a 
$(-1)$-shifted symplectic derived structure, hence their underlying moduli  scheme has an algebraic $d$-critical locus structure. 
Thus the algebraic $d$-critical locus of Joyce provides the classical schematical framework for the moduli space of stable simple complex over smooth Calabi-Yau threefolds. 

To define the algebraic $d$-critical scheme, we first recall the following theorem in \cite{Joyce}:

\begin{thm}\label{property:local:data}(\cite{Joyce})
Let $X$ be a $\kappa$-scheme, which is locally of finite type. Then there exists a sheaf $\sS_{X}$ of $\kappa$-vector spaces on $X$, unique up to canonical isomorphism, which is uniquely characterized by the following two properties:

(i) Suppose that $R\subseteq X$ is Zariski open, $U$ is a smooth $\kappa$-scheme, and $i: R\hookrightarrow U$ is a closed embedding. Then there is an exact sequence of sheaves of $\kappa$-vector spaces on $R$:
$$0\rightarrow I_{R,U}\longrightarrow i^{-1}(\oO_{U})\stackrel{i^{\#}}{\longrightarrow}\oO_{X}|_{R}\rightarrow 0,$$
where $\oO_{X}, \oO_{U}$ are the structure sheaves of $X$ and $U$, and $i^{\#}$ is the morphism of sheaves over $R$. 
There is also an exact sequence of sheaves of $\kappa$-vector spaces over $R$:
$$0\rightarrow\sS_{X}|_{R}\stackrel{\iota_{R,U}}{\longrightarrow}\frac{i^{-1}(\oO_{U})}{I^{2}_{R,U}}\stackrel{d}{\longrightarrow}
\frac{i^{-1}(T^*U)}{I_{R,U}\cdot i^{-1}(T^*U)}$$
where $d$ maps $f+I_{R,U}^2$ to $df+I_{R,U}\cdot i^{-1}(T^*U)$. 

(ii) If $R\subseteq S\subseteq X$ are Zariski open, and $U, V$ are smooth $\kappa$-schemes, and 
$$i: R\hookrightarrow U$$
$$j: S\hookrightarrow V$$
are closed embeddings. Let 
$$\Phi: U\to V$$ be a morphism with $\Phi\circ i=j|_{R}: R\to V$.  Then the following diagram of sheaves on $R$ commutes:
\begin{equation}\label{diagram:Algebraic:Dcritical:Locus}
\begin{CD}
0 @ >>>\sS|_{R}@ >{\iota_{S,V}|_{R}}>> \frac{j^{-1}(\oO_{V})}{I^2_{S,V}}|_{R}@ >{d}>>
\frac{j^{-1}(T^*V)}{I_{S,V}\cdot j^{-1}(T^*V)}|_{R}@
>>> 0\\
&& @VV{\id}V@VV{i^{-1}(\Phi^{\#})}V@VV{i^{-1}(d\Phi)}V \\
0@ >>> \sS_{X}|_{R} @ >{\iota_{R,U}}>>\frac{i^{-1}(\oO_{U})}{I^2_{R,U}}@ >{d}>> \frac{i^{-1}(T^*U)}{I_{R,U}\cdot i^{-1}(T^*U)} @>>> 0.
\end{CD}
\end{equation}
Here $\Phi: U\to V$ induces
$$\Phi^{\#}: \Phi^{-1}(\oO_{V})\to\oO_{U}$$
on $U$, and we have:
\begin{equation}\label{map1}
i^{-1}(\Phi^{\#}): j^{-1}(\oO_{V})|_{R}=i^{-1}\circ\Phi^{-1}(\oO_{V})\to i^{-1}(\oO_{U}),
\end{equation}
a morphism of sheaves of $\kappa$-algebras on $R$.  As 
$\Phi\circ i=j|_{R}$, then (\ref{map1}) maps to $I_{S,V}|_{R}\to I_{R,U}$,  and 
$I^2_{S,V}|_{R}\to I_{R,U}^{2}$.  Thus (\ref{map1}) induces the morphism in the second column of (\ref{diagram:Algebraic:Dcritical:Locus}). Similarly,  $d\Phi: \Phi^{-1}(T^*V)\to T^*U$ induces the third column of 
(\ref{diagram:Algebraic:Dcritical:Locus}).
\end{thm}

According to \cite{Joyce}, there is a natural decomposition 
$$\sS_{X}=\sS_{X}^0\oplus \kappa_{X}$$
and $\kappa_{X}$ is the constant sheaf on $X$ and $\sS_{X}\subset \sS_{X}$ is the kernel of the composition:
$$\sS_{X}\to\oO_{X}\stackrel{i_{X}^{\#}}{\longrightarrow}\oO_{X^{\red}}$$
with $X^{\red}$ the reduced $\kappa$-scheme of $X$, and $i_{X}: X^{\red}\hookrightarrow X$ the inclusion. 

\begin{defn}\label{algebraic:d:critical:locus}
An {\em algebraic d-critical scheme} over the field $\kappa$ is a pair $(X,s)$, where $X$ is a $\kappa$-scheme, locally of finite type, and 
$s\in H^0(\sS_{X}^0)$ for $\sS_{X}^{0}$ in Theorem \ref{property:local:data}. These data satisfy the following conditions:
for any $x\in X$, there exists a Zariski open neighbourhood $R$ of $x$ in $X$, a smooth $\kappa$-scheme $U$, a regular function 
$f: U\to\kappa$, and a closed embedding $i: R\hookrightarrow U$, such that $i(R)=\Crit(f)$ as $\kappa$-subschemes of $U$, 
and $\iota_{R,U}(s|_{R})=i^{-1}(f)+I^2_{R,U}$.  We call the quadruple $(R,U,f,i)$ a {\em critical chart} on $(X,s)$. 
\end{defn}

Some properties of $(X,s)$ are as follows:

\begin{thm}\cite{Joyce}
Let $(X,s)$ be an algebraic $d$-critical scheme, and 
$(R,U,f,i), (S, V,g,j)$ be critical charts on $(X,s)$.  Then for each $x\in R\cap S\subset X$ there exists subcharts 
$$(R^\prime, U^\prime, f^\prime, i^\prime)\subseteq (R, U, f, i),$$
$$(S^\prime, V^\prime, g^\prime, j^\prime)\subseteq (S, V, g, j)$$
with $x\in R^\prime\cap S^\prime\subseteq X$, a critical chart $(T, W,h,k)$ on $(X,s)$, and embeddings
$$\Phi: (R^\prime, U^\prime, f^\prime, i^\prime)\hookrightarrow (T, W, h, k)$$ and 
$$\Psi: (S^\prime, V^\prime, g^\prime, j^\prime)\hookrightarrow (T, W, h, k).$$
\end{thm}

We introduce the canonical line bundle of $(X,s)$:

\begin{thm}\label{Canonical:line:bundle:Xs}\cite[Theorem 2.28]{Joyce}
Let $(X,s)$ be an algebraic $d$-critical scheme, and $X^{\red}\subset X$ the associated reduced $\kappa$-scheme. 
Then there exists a line bundle $K_{X,s}$ on $X^{\red}$ which we call the {\em canonical Line bundle} of $(X,s)$, that is natural up to  canonical isomorphism, and is characterized by the following properties:

(i) If $(R,U,f,i)$ is a critical chart on $(X,s)$, there is a natural isomorphism 
$$\iota_{R,U,f,i}: (K_{X,s})|_{R^{\red}}\to i^*(K_{U}^{\otimes 2})|_{R^{\red}}$$
where $K_U$ is the canonical line bundle of $U$.

(ii) Let $\Phi: (R,U,f,i)\hookrightarrow  (S, V,g,j)$ be an embedding of critical charts on $(X,s)$. Then there is an 
isomorphism of line bundles on $\Crit(f)^{\red}$:
$$J_{\Phi}: (K_{U}^{\otimes 2})|_{\Crit(f)}\stackrel{\cong}{\longrightarrow} \Phi|_{\Crit(f)^{\red}}^{*}(K_{V}^{\otimes 2}).$$
Since $i: R\to \Crit(f)$ is an isomorphism as schemes with $\Phi\circ i=j|_{R}$, this gives 
$$i|_{R^{\red}}^{*}(J_{\Phi}): i^{*}(K_{U}^{\otimes 2})|_{R^{\red}}\stackrel{\cong}{\longrightarrow} j^{*}(K_{V}^{\otimes 2})|_{R^{\red}},
$$
and we have:
$$\iota_{S,V,g,j}|_{R^{\red}}=i|_{R^{\red}}^{*}(J_{\Phi})\circ \iota_{R,U,f,i}: (K_{X,s})|_{R^{\red}}
\to j^{*}(K_{V}^{\otimes 2})|_{R^{\red}}.$$
\end{thm}

\begin{defn}\label{oriented:d:critical:locus}
Let $(X,s)$ be an algebraic $d$-critical scheme, and $K_{X,s}$ the canonical line bundle of $(X,s)$. 
An {\em orientation} on $(X,s)$ is a choice of square root line bundle $K_{X,s}^{1/2}$ for $K_{X,s}$ on $X^{\red}$. 
I.e., an orientation of $(X,s)$ is a line bundle $L$ over $X^{\red}$ and an isomorphism $L^{\otimes 2}=L\otimes L\cong K_{X,s}$. 
A $d$-critical scheme with an orientation will be called an oriented $d$-critical scheme.
\end{defn}

Bussi, Brav and Joyce \cite{BBJ} prove the following interesting result:  Let $(\bX, \omega)$ be a $(-1)$-shifted symplectic derived scheme over $\kappa$ in the sense of \cite{PTVV}, and let $X:=\bt_0(\bX)$ be the associated classical $\kappa$-scheme of 
$\bX$. Then $X$ naturally extends to an algebraic $d$-critical scheme $(X,s)$. The canonical line bundle $K_{X,s}\cong \det(\ll_{\bX})|_{X^{\red}}$ is the determinant line bundle of the cotangent complex $\ll_{\bX}$ of $\bX$.

One of the applications of the $(-1)$-shifted symplectic derived scheme or stack is on moduli problems.  Let $Y$ be a smooth Calabi-Yau threefold over $\kappa$, and $X$ a classical moduli scheme of simple coherent sheaves in $\Coh(Y)$, the abelian category of coherent sheaves on $Y$.  Then in \cite{PTVV}, the authors prove that there is a natural $(-1)$-shifted derived scheme structure 
$\mathbf{X}$ on the moduli space $X$, such that if 
$$i: X\hookrightarrow \mathbf{X}$$
is the inclusion, then the pullback $i^{*}\ll_{\mathbf{X}}$ of the cotangent complex of $\mathbf{X}$ is a perfect obstruction theory of 
$X$, 
thus from the result in \cite{BBJ}, $X$ has an algebraic $d$-critical locus structure.

\begin{example}\label{example_d_critical_scheme_cone}
Consider $X=\Crit(f)$ to be the critical locus for the function 
$$f=x^2y: U=\aaa_{\kappa}^2\to \aaa_{\kappa}^1,$$ 
and then 
$X=\spec(\kappa[x,y]/(xy, x^2))$.  Let $i: X\hookrightarrow U$ be the inclusion. Then $X$ is naturally a $d$-critical scheme, with 
$$\sS_X=\ker\left(\frac{i^{-1}(\oO_U)}{I_X^2}\stackrel{d}{\longrightarrow}\frac{i^{-1}(\Omega_U)}{I_X\cdot \Omega_U} \right).$$
The sheaf $\sS_X=\sS_X^{0}\oplus \kappa_{X}$, where on $X^{\red}$ the function $f$ can be written down as $f=f^0+c$ and $c$ is locally constant.  From Joyce's explanation, $\sS_X^0$ is the coherent sheaf that locally remembers the closed one form $df$. 

We take $\overline{X}=\proj R[y_0:y_1]/(x^2, xy_0)$, where $R=\kappa[x]/(x^2)$.  Then $\overline{X}$ is a compactification of $X$, which is a $\mathbb{P}^1=\proj\kappa[y_0:y_1]$ with a non-reduced point $0\in\mathbb{P}^1$. 
Let $\infty\in \mathbb{P}^1$ be the infinity point.  We explain that $\overline{X}$ is also a $d$-critical scheme. 
Since $\overline{X}\setminus \{\infty\}=X$, then we have a $d$-critical chart:
$$(X, U, f, i)$$
as above.  The section $s\in \sS_{\overline{X}}$ satisfies that 
$$\iota(s|_{X})=f+I_X^2$$
where $\iota: \sS_{\overline{X}}|_{X}\to \frac{i^{-1}(\oO_U)}{I_X^2}$ is the inclusion. 

Since $\overline{X}\setminus \{0\}\cong \aaa_{\kappa}^1$, then we have a $d$-critical chart:
$$(\overline{X}\setminus \{0\}, \aaa_{\kappa}^1, 0, j).$$
Then $\sS_{\overline{X}}|_{\overline{X}\setminus \{0\}}=0$. 
These  two $d$-critical charts glue to give the $d$-critical scheme $\overline{X}$.  As proved in \cite{Joyce}, 
$\overline{X}$ is non orientable. 
\end{example}

\subsection{Symmetric semi-perfect obstruction theory of $d$-critical schemes}\label{subsec_SPOT_d_critical_schemes}

In \cite{KL1}, Kiem-Li has proved that there exists a symmetric semi-perfect obstruction theory on a virtual critical manifold defined in \cite{KL1}.  Taking as analytic spaces, virtual critical manifolds are the same as $d$-critical analytic spaces.  We provide a proof here for $d$-critical schemes and all credits of the result  belong to Kiem-Li in 
\cite{KL1}.

\begin{thm}\label{thm_d_critical_scheme_symmetric_semi_obstruction_theory}
Let $(X,s)$ be a $d$-critical scheme in the sense of \cite{Joyce}. Then $(X,s)$ admits a symmetric semi-perfect obstruction theory. 
\end{thm}
\begin{proof}
The $d$-critical scheme $(X,s)$ is covered by the $d$-critical charts $(X_\alpha, U_\alpha, f_\alpha, i_\alpha)_{\alpha\in \Lambda}$, where 
$$i_\alpha: X_\alpha\hookrightarrow X$$
$$f_\alpha: U_\alpha\to \aaa_{\kappa}^1$$
such that $\Crit(f_\alpha)\cong X_\alpha$.  
Locally on each critical chart, there exists a symmetric obstruction theory
\begin{equation}\label{eqn_local_obstruction}
\xymatrix{
E_{\alpha}^{\bullet}:\ar[d]_{\phi} & [T_{U_{\alpha}}|_{X_\alpha}\ar[r]^--{d\circ df^{\vee}}\ar[d]_{df^{\vee}}&
\Omega_{U_\alpha}|_{X_\alpha}\ar[d]^{=}] \\
L_{X_\alpha}^{\bullet}: &[ I_\alpha/I_\alpha^2\ar[r]^{d}&\Omega_{U_\alpha}|_{X_\alpha}]
}
\end{equation}
For any $\alpha, \beta$, and two $d$-critical charts: 
$$(X_\alpha, U_\alpha, f_\alpha, i_\alpha), \quad  (X_\beta, U_\beta, f_\beta, i_\beta),$$
let 
$X_{\alpha\beta}=X_\alpha\cap \beta_\beta$.  From \cite{Joyce}, for any $x\in X_{\alpha\beta}$, there exist 
sub-critical charts 
$$(X^\prime_\alpha, U^\prime_\alpha, f^\prime_\alpha, i^\prime_\alpha)\subset (X_\alpha, U_\alpha, f_\alpha, i_\alpha)$$ and 
$$(X^\prime_\beta, U^\prime_\beta, f^\prime_\beta, i^\prime_\beta)\subset (X_\beta, U_\beta, f_\beta, i_\beta)$$
such that for any $x\in X^\prime_{\alpha}\cap X^\prime_{\beta}\subset X$, there exists a critical chart $(T, W, h, k)$ such that 
$$\Phi:  (X^\prime_\alpha, U^\prime_\alpha, f^\prime_\alpha, i^\prime_\alpha)\hookrightarrow (T, W, h, k)$$ and 
$$\Psi: (X^\prime_\beta, U^\prime_\beta, f^\prime_\beta, i^\prime_\beta)\hookrightarrow  (T, W, h, k)$$
are embeddings of the critical charts. 

For the local symmetric obstruction theory $\phi_\alpha: E_\alpha^\bullet\to L_{X_\alpha}^{\bullet}$ on $X_\alpha$, we show:
\begin{enumerate}
\item For each embedding of critical charts:
$$\Phi:  (X^\prime_\alpha, U^\prime_\alpha, f^\prime_\alpha, i^\prime_\alpha)\hookrightarrow (X_\alpha, U_\alpha, f_\alpha, i_\alpha)$$
we have a morphism
$$\Phi_{\star}: E_{\alpha}^{\prime\bullet} \to E_{\alpha}^{\bullet}|_{X_\alpha^\prime}$$ such that 
it induces an isomorphism on the bundle stacks 
$h^1/h^0(E_{\alpha}^{\prime\bullet\vee})\cong h^1/h^0(E_{\alpha}^{\bullet\vee}|_{X_\alpha^\prime})$.
\item For two embeddings of critical charts: 
$$\Phi:  (X^{\prime\prime}_\alpha, U^{\prime\prime}_\alpha, f^{\prime\prime}_\alpha, i^{\prime\prime}_\alpha)\hookrightarrow (X^\prime_\alpha, U^\prime_\alpha, f^\prime_\alpha, i^\prime_\alpha)$$ and 
$$\Psi:  (X^\prime_\alpha, U^\prime_\alpha, f^\prime_\alpha, i^\prime_\alpha)\hookrightarrow (X_\alpha, U_\alpha, f_\alpha, i_\alpha)$$
we have $(\Psi\circ\Phi)_{\star}=\Psi_{\star}|_{X_{\alpha}^{\prime\prime}}\circ \Phi_{\star}$. 
\end{enumerate}

The claim (2) is from (1).   We prove the property (1).  The embedding $\Phi$ induces the following commutative diagram:
\[
\xymatrix{
X_{\alpha}^\prime\ar@{^{(}->}[r]^{\Phi}\ar@{^{(}->}[d]_{i_\alpha^\prime}& X_\alpha\ar@{^{(}->}[d]^{i_\alpha}\\
U_\alpha^\prime\ar@{^{(}->}[r]^{\Phi}& U_\alpha
}
\]
From \cite[Proposition 2.2.2, Proposition 2.2.3]{Joyce}, shrink $U_\alpha^\prime$ if necessary,  the critical charts satisfy the following properties:
$$U_\alpha\cong U_{\alpha}^\prime\times \aaa_{\kappa}^n$$
where $n=\dim U_\alpha-\dim U_{\alpha}^\prime$, such that 
$$f_\alpha=f_{\alpha}^\prime+ z_1^2+\cdots+z_n^2.$$
Then $E_\alpha^\bullet$ is isomorphic to 
$$E_\alpha^\bullet\cong \Big[T_{U^\prime\times\aaa_{\kappa}^n}|_{X_\alpha}\to \Omega_{U^\prime\times\aaa_{\kappa}^n}|_{X_\alpha}\Big]$$
and the bundle stack 
\begin{align*}
h^1/h^0(E_\alpha^{\bullet\vee})|_{X_{\alpha^\prime}}&=[\Omega_{U_\alpha}|_{X_\alpha}/T_{U_\alpha}|_{X_\alpha}]\\
&=[\Omega_{U^\prime_\alpha\times\aaa_{\kappa}^n}|_{X^\prime_\alpha}/T_{U^\prime_\alpha\times\aaa_{\kappa}^n}|_{X^\prime_\alpha}]\\
&=[\Omega_{U^\prime_\alpha}|_{X^\prime_\alpha}/T_{U^\prime_\alpha}|_{X^\prime_\alpha}]
\end{align*}
where the last equality is true since $[\Omega_{\aaa_{\kappa}^n}/T_{\aaa_{\kappa}^n}]$ is trivial.  This fact can be proved by the following arguments.
Since $h=z_1^2+\cdots+z_n^2$, then the morphism $dh^\vee\circ d: T_{\aaa_{\kappa}^n}\to 
\Omega_{\aaa_{\kappa}^n}$ is an isomorphism, and the morphism send the basis $\{\frac{\partial}{\partial z_1},\cdots,\frac{\partial}{\partial z_n}\}$ to 
$\{dz_1,\cdots, dz_n\}$.  Therefore the quotient $[\aaa_{\kappa}^n/\aaa_{\kappa}^n]$ is trivial. 
\begin{rmk}
We can understand this bundle stack $[\aaa_{\kappa}^n/\aaa_{\kappa}^n]$ as follows.  For simplicity, let $n=1$,  it is known that the stack 
$[\G_m/\G_m]$ where $\lambda\in\G_m$ acts on $\G_m$ by the multiplication.  The stack is just a point since the action has only one orbit.   Then we understand the bundle stack $[\aaa^1_{\kappa}/\aaa^1_{\kappa}]$ is trivial since there is also only one orbit of $\aaa^1_{\kappa}$ by the action of $\aaa^1_{\kappa}$. 
\end{rmk}

We then use (1) and (2) to prove the following: let $(X_\alpha, U_\alpha, f_\alpha, i_\alpha)$ and  
$(X_\beta, U_\beta, f_\beta, i_\beta)$
be two critical charts. Then shrinking $X_\alpha$, $X_\beta$ if necessary, 
\begin{equation}\label{eqn_bundle_stack_intersection}
h^1/h^0(E_\alpha^\vee)|_{X_\alpha\cap X_\beta}\cong h^1/h^0(E_\beta^\vee)|_{X_\alpha\cap X_\beta}.
\end{equation}

From \cite[Theorem 2.20]{Joyce}, for any $x\in X_\alpha\cap X_\beta$, shrinking $X_\alpha, X_\beta$ if necessary, there exists subcritical charts 
$$(X_\alpha^\prime, U_\alpha^\prime, f_\alpha^\prime, i_\alpha^\prime)\hookrightarrow (T_\gamma, W_\gamma, h_\gamma, k_\gamma)$$ and 
$$(X_\beta^\prime, U_\beta^\prime, f_\beta^\prime, i_\beta^\prime)\hookrightarrow (T_\gamma, W_\gamma, h_\gamma, k_\gamma).$$
Let $E_\gamma^\bullet$ be the symmetric obstruction theory of $(T_\gamma, W_\gamma, h_\gamma, k_\gamma)$. Then from property (1), 
$$h^1/h^0(E_\gamma^{\bullet\vee})|_{X_\alpha^\prime}\cong h^1/h^0(E_\alpha^{\bullet\vee})$$
and 
$$h^1/h^0(E_\gamma^{\bullet\vee})|_{X_\beta^\prime}\cong h^1/h^0(E_\beta^{\bullet\vee}).$$
Since $x\in X^\prime_\alpha\cap X^\prime_\beta\subset X_\alpha\cap X_\beta\subset X$, we may choose $X_\alpha^\prime$, $X_\beta^\prime$ small enough such that 
we have subcritical charts:
$$(X_\alpha^\prime\cap X_\beta^\prime, U_\alpha^\prime\cap U_\beta^\prime, 
f_\alpha^\prime|_{U_\alpha^\prime\cap U_\beta^\prime}, i_\alpha^\prime|_{X_\alpha^\prime\cap X_\beta^\prime})\hookrightarrow (T_\gamma, W_\gamma, h_\gamma, k_\gamma)$$ and 
$$(X_\alpha^\prime\cap X_\beta^\prime, U_\alpha^\prime\cap U_\beta^\prime, 
f_\beta^\prime|_{U_\alpha^\prime\cap U_\beta^\prime}, i_\beta^\prime|_{X_\alpha^\prime\cap X_\beta^\prime})\hookrightarrow (T_\gamma, W_\gamma, h_\gamma, k_\gamma)$$
such that we get:
$$h^1/h^0(E_\gamma^{\bullet\vee})|_{X_\alpha^\prime\cap X_\beta^\prime}\cong 
h^1/h^0(E_\alpha^{\bullet\vee})|_{X_\alpha^\prime\cap X_\beta^\prime}$$ and 
$$h^1/h^0(E_\gamma^{\bullet\vee})|_{X_\alpha^\prime\cap X_\beta^\prime}\cong 
h^1/h^0(E_\beta^{\bullet\vee})|_{X_\alpha^\prime\cap X_\beta^\prime}$$
and 
$$h^1/h^0(E_\alpha^{\bullet\vee})|_{X_\alpha^\prime\cap X_\beta^\prime}\cong 
h^1/h^0(E_\beta^{\bullet\vee})|_{X_\alpha^\prime\cap X_\beta^\prime}.$$
The obstruction sheaf $h^1(E_\alpha^{\bullet\vee})=\Omega_{X_\alpha}$ for all $\alpha$ glue to give the obstruction sheaf $\Omega_{X}$.  Hence the Condition (1) in the definition of symmetric semi-perfect obstruction theory is proved. 

Finally we prove that for two local symmetric obstruction theories:
$\phi_\alpha: E_\alpha^\bullet\to L_{X_\alpha}^\bullet$, and $\phi_\beta: E_\beta^\bullet\to L_{X_\beta}^\bullet$, 
$\phi_{\alpha}|_{X_\alpha\cap X_\beta}$ is $\nu$-equivalent to $\phi_{\beta}|_{X_\alpha\cap X_\beta}$. 
Let $p\in X_\alpha\cap X_\beta=X_{\alpha\beta}\subset X$ be a closed point, and let $\overline{T}=\spec \overline{A}$ be an Artin local ring with maximal ideal $\mathfrak{m}_{\overline{T}}$.  Let $I$ be an ideal in $\overline{A}$ such that $I\cdot \mathfrak{m}_{\overline{T}}=0$ and let 
$$T=\overline{T}/I  \quad (A=\overline{A}/I).$$
Consider the following diagram:
\begin{equation}\label{eqn_diagram_nu_equivalent}
\xymatrix{
\spec(A)\ar[r]^{g}\ar@{^{(}->}[d]& X_{\alpha\beta}\ar[d]\ar[r]& X_\alpha\ar@{^{(}->}[r]^{i_\alpha}& U_\alpha\\
\spec(\overline{A})& X_\beta\ar@{^{(}->}[r]^{i_\beta}& U_\beta.
}
\end{equation}
Since $U_\alpha, U_\beta$ are smooth, there are morphisms 
\begin{equation}\label{eqn_star_1}
g_\alpha: \spec(\overline{A})\to U_\alpha; \quad  g_\beta: \spec(\overline{A})\to U_\beta
\end{equation}
extending the morphisms in (\ref{eqn_diagram_nu_equivalent}). 
The functions $f_\alpha, f_\beta$ give 
\begin{equation}\label{eqn_star_2}
df_\alpha:  U_\alpha\to \Omega_{U_\alpha}; \quad  df_\beta: U_\beta\to \Omega_{U_\beta}.
\end{equation}
Let 
\begin{equation}\label{eqn_star_3}
\rho_\alpha: I\otimes_{\kappa}\Omega_{U_\alpha}|_{p}\to I\otimes_{\kappa}\Omega_{X_\alpha}|_{p}; \quad. 
\rho_\beta: I\otimes_{\kappa}\Omega_{U_\beta}|_{p}\to I\otimes_{\kappa}\Omega_{X_\beta}|_{p}.
\end{equation}
be the canonical morphisms.  Then from (\ref{eqn_star_1}), (\ref{eqn_star_2}) and (\ref{eqn_star_3}), we have:
$$\Ob(\phi_\alpha, g, T, \overline{T})_{p}=\rho_{\alpha}(df_{\alpha}\circ g_\alpha|_{p})\in I\otimes_{\kappa}\Omega_{X_{\alpha\beta}}|_{p}$$ and 
$$\Ob(\phi_\beta, g, T, \overline{T})_{p}=\rho_{\beta}(df_{\beta}\circ g_\beta|_{p})\in I\otimes_{\kappa}\Omega_{X_{\alpha\beta}}|_{p}$$
by \cite[Lemma 1.28]{KL}. So we need to show that 
$$\Ob(\phi_\alpha, g, T, \overline{T})_{p}=\Ob(\phi_\beta, g, T, \overline{T})_{p}.$$

We write down the detail composition morphisms.  
$$\spec(\overline{A})\stackrel{g_\alpha}{\longrightarrow}U_{\alpha}\stackrel{df_\alpha}{\longrightarrow}
\Omega_{U_\alpha}\stackrel{\rho_\alpha}{\longrightarrow}I\otimes_{\kappa}\Omega_{X_\alpha},$$
$$\spec(\overline{A})\stackrel{g_\beta}{\longrightarrow}U_{\beta}\stackrel{df_\beta}{\longrightarrow}
\Omega_{U_\beta}\stackrel{\rho_\beta}{\longrightarrow}I\otimes_{\kappa}\Omega_{X_\beta},$$
and the obstruction space is the composition morphisms restricted to $U_{\alpha\beta}, p\in X_{\alpha\beta}$. 
In between we can insert the composition morphism 
$$\spec(\overline{A})\stackrel{g_\gamma}{\longrightarrow}U_{\gamma}\stackrel{df_\gamma}{\longrightarrow}
\Omega_{U_\gamma}\stackrel{\rho_\gamma}{\longrightarrow}I\otimes_{\kappa}\Omega_{X_\gamma},$$
where when restricted to the point $p$, they are all the same. Therefore 
$$\Ob(\phi_\alpha, g, T, \overline{T})_{X_{\alpha\beta}, p}=\Ob(\phi_\beta, g, T, \overline{T})_{X_{\alpha\beta}, p}$$
and we are done. 
\end{proof}

\subsection{Application to virtual signed Euler characteristics} \label{sec_Virtual_Euler}

In this section we apply the former result to the dual obstruction sheaf for a perfect obstruction theory without assuming derived schemes as studied in \cite{JT}.  Note that this result is already proved in \cite{Kiem}.  We also talked about one compactification of the dual obstruction sheaf cone. 

\subsubsection{The dual obstruction sheaf}

For the consistence of notations, we follow \cite{JT}.  Let $X$ be a scheme with a perfect obstruction theory $E_X^\bullet$ in the sense of \cite{BF}, \cite{LT}.   The obstruction sheaf $\Ob_X:=h^1(E_X^{\bullet\vee})$.  Let
$$N:=\spec (\Sym^\bullet \Ob_X)$$
be the abelian cone corresponding to this  obstruction sheaf.  It is a cone over $X$ together with a $\cc^*$-action by scaling the fibers of $N$.  Let $\pi: N\to X$ be the projection. 

From \cite{JT}, 
given a locally free resolution $E_0\stackrel{\phi}{\longrightarrow} E_1\to \Ob_X\to 0$. 
Let $\tau$ be the tautological section of $\pi_{E_1}^*E^{-1}$.
Therefore,  $N=C(\Ob_X)$ is cut out of Tot$\;(E^{-1})|_X$ by the section $\pi_{E_1}^*(\phi)^*(\tau)$ of 
$\pi_{E^{-1}}E^0$.

\subsubsection*{Local model}

Locally we choose a presentation of $(X,E_X^\bullet)$ as the zero locus of a section $s$ of a vector bundle $E\to A$ over a smooth scheme $A$, such that the  complex
$$
\Big[T_A|_X\stackrel{ds}{\longrightarrow}E|_X\Big] \quad\text{is quasi-isomorphic to }\quad
(E_X^\bullet)^\vee.
$$
Let $\tau$ be the tautological section of $\pi_{E}^*E^*$.
Therefore,  $N=C(\Ob_X)$ is cut out of Tot$\;(E^*)|_X$ by the section $\pi_E^*(ds)^*(\tau)$ of $\pi_E^*\Omega_A|_X$. In turn Tot$\;(E^*)|_X$ is cut out of Tot$\;(E^*)$ by $\pi_E^*s$. Therefore the ideal of $N$ in the smooth ambient space Tot$\;(E^*)$ is
\begin{equation}\label{ideal}
\big(\pi_E^*s,\,\pi_E^*(Ds)^*(\tau)\big),
\end{equation}
where we have chosen any holomorphic connection $D$ on $E\to A$ by shrinking $A$ if necessary.

Thinking of the section $s$ of $E\to A$ as a linear function $\widetilde s$ on the fibres of Tot$\;(E^*)$, we find that its critical locus is $N$.  From \cite[Proposition 2.6]{JT}, the function $\widetilde s$ is:
$$
\widetilde s\colon\mathrm{Tot}\;(E^*)\to\aaa_{\kappa}^1\;.
$$
where 
$$\widetilde s=\big\langle\pi_E^*s,\tau\big\rangle=\pi_E^*s^*(\tau)$$
in terms of the tautological section $\tau$ of $\pi_E^*E^*$.  Then
\begin{equation}\label{critical_locus}
N=\Crit(\widetilde s).
\end{equation}
We work in local coordinates $x_i$ for $A$. Trivialising $E$ (with rank $r$) with a basis of sections $e_j$, we get a dual basis $f_j$ for $E^*$ and coordinates $y_j$ on the fibers of Tot$\;(E^*)$.

Then we can write $s=\sum_js_je_j,\ \tau=\sum_{j=1}^ry_jf_j$ and
$$\widetilde s=\sum_{j=1}^r s_jy_j.$$
Therefore
$$
d\;\widetilde s=\sum_jy_jds_j+\sum_js_jdy_j
=\big\langle\tau,\pi_E^*Ds\big\rangle+\sum_js_jdy_j
$$
with zero scheme defined by the ideal
$$
\big(\pi_{E^*}^*(Ds)^*(\tau),\,\pi_E^*s_1,\,\pi_E^*s_2,\,\ldots\big).
$$
This is the same as \eqref{ideal}.

For the scheme $N$, from \cite[Theorem 2.1]{Joyce}, there exists a unique coherent sheaf $\sS_N$, such that 
in the local model $R\subset N$ such that $N_\alpha\cong Crit(\widetilde s)$ for a regular function $\widetilde s:  \widetilde A\to \cc$, then 
there is a section $s\in \sS_N$ such that 
$$\iota(s|_{N_\alpha})=\widetilde s+I_{N_\alpha}^2.$$
Therefore we get a critical chart $(R, \widetilde A, \widetilde s, i)$, where $i: R\hookrightarrow \widetilde A$ is the inclusion.   Since $N$ is covered by open subschemes $R$ such that they give the local models for $N$,  in \cite{Jiang2017}, we show:

\begin{prop}(\cite[Proposition 2.5]{Jiang2017})
$(N,s)$ is a $d$-critical scheme. 
\end{prop}

Then from Theorem \ref{thm_d_critical_scheme_symmetric_semi_obstruction_theory}, $N$ admits a symmetric semi-perfect obstruction theory $\phi=\{\phi_\alpha, N_\alpha, E_\alpha, \psi_{\alpha\beta}\}_{\alpha\in\Lambda}$
with obstruction sheaf $\Omega_N$. 

In \cite{Kiem}, Kiem  defined the cosection  of a semi-perfect obstrution theory and the cosection here  
$$\sigma: \Omega_N\to \oO_N$$
 is constructed 
by taking the dual of the vector field 
$$v_x=\frac{d}{d\lambda}(\lambda\cdot x)|_{\lambda=1}$$
of the $\G_m$-action.  We fix a global embedding $N\hookrightarrow \widetilde A$ into a higher dimensional smooth scheme $\widetilde A$ such that $\widetilde A\to A$ is a vector bundle over a smooth scheme $A$ and $X\hookrightarrow A$ is the global immersion of $X$, see \cite[\S 4]{JT}.  We have a cartesian diagram 
\begin{equation}\label{diagram__N_Cone}
\xymatrix{
C\ar@{^{(}->}[r]\ar[d]& \Omega_{\widetilde A}|_{N}\ar[d]\\
cv\ar@{^{(}->}[r]& \Omega_N
}
\end{equation}
where $cv$ is the coarse moduli sheaf of the intrinsic normal cone, and $C$ is the unique lifting cone making the diagram commute.  The important key point is that 
$$C\hookrightarrow \Omega_{\widetilde A}|_{N}\subset \Omega_{\widetilde A}$$
is Lagrangian inside $\Omega_{\widetilde A}$, which is naturally a symplectic manifold. From \cite{KL}, 
$$C\subset  \Omega_{\widetilde A}|_{X}\sqcup\ker( \Omega_{\widetilde A}|_{N}\to \oO_N).$$
Taking a small perturbation $\xi$ is the zero section $\widetilde A$ of $\Omega_{\widetilde A}$ such that $\xi\cap C$ only supports on $M$.  Then Kiem-Li constructed the localized virtual cycle 
$$[N]^{\virt}_{\loc}:=\xi\cap C \in A_0(X).$$
\begin{thm}
We have:
$$\int_{[N]^{\virt}_{\loc}}1=\chi(N, \nu_N).$$
\end{thm}
\begin{proof}
The proof is nearly the same as in \cite[\S 4]{JT}.   Note that the essential point is that $C\subset \Omega_{\widetilde A}$ is Lagrangian, which is still true for symmetric  semi-perfect obstruction theory $\phi=\{\phi_\alpha, N_\alpha,  E_\alpha, \psi_{\alpha\beta}\}_{\alpha\in \Lambda}$ since it is a local property. 
\end{proof}

\subsubsection{$\G_m$-localized invariants vs Fantechi-Goeschett/Ciocan-Fortanine-Kapranov virtual Euler characteristics}

We also generalize the invariants $(1), (2)$ in \cite{JT} to $N$ admitting a symmetric semi-perfect obstruction theory,  
as studied in \cite{Kiem}. 

Let $(X, E_X^\bullet)$ be a perfect obstruction theory on a scheme $X$.  The virtual Euler number is defined as:
$$e_{\virt}(X)=\int_{[X]^{\virt}}c_{\vd}((E_X^\bullet)^{\vee})$$
where $(E_X^\bullet)^{\vee}$ is taken as the virtual tangent bundle on $X$. We use the signed version:
$$e_{1}(X)=\int_{[X]^{\virt}}c_{\vd}(E_X^\bullet)$$
which is deformation invariant. 

On the other hand, since $\G_m$ acts on $N$ (the abelian cone), one can apply Graber-Pandharipande's virtual localization on the virtual cycle of $N$ (for semi-perfect obstruction theory on $N$, \cite{Kiem} generalizes it to this situation). We have, over the $\G_m$ fixed $X$, 
$$E_\alpha^{\bullet}|_{X}\cong E_X^\bullet\oplus (E_X^\bullet)^{\vee}\otimes \mathbf{t}^{-1}[1]$$
for each $\alpha$, where $\mathbf{t}$ denotes the standard weight one representation of $\G_m$.  The virtual normal bundle is given by the dual of the second summand
$$\mathcal{N}^{\virt}\cong E_X^\bullet\otimes \mathbf{t}[-1].$$
The the rest if the same as in \cite[\S 3.2]{JT}, 
$$e_2(X):=\int_{[X]^{\virt}}\frac{1}{e(\mathcal{N}^{\virt})}=e_1(X). $$
Therefore 
\begin{thm}
The Graber/Pandharipande localized invariants $e_2(X)$ is the same as the signed Euler number $e_1(X)$. 
\end{thm}

\subsection{A compactification of $N$}

We provide a compactification $\overline{N}$ of the abelian cone $N$.

The natural compactification $\overline{N}$ of $N$ is defined as:
\begin{equation}\label{defnBarN}
\overline{N}:=\overline{C}(\Ob_X)=\Proj(\Sym^\bullet(\Ob_X\oplus \oO_X))\stackrel{\overline{\pi}}{\longrightarrow} X.
\end{equation}
Then $\overline{N}$ is a projective cone over $X$ containing $N$ as an open locus. 
Let $D_\infty^{N}:=\Proj(\Sym^{\bullet}\Ob_X)$ be the infinity divisor of $\overline{N}$.

We review the construction of the projective cone $\overline{N}$.  Writing $E_X^\bullet$ as $E^{-1}\to E^0$, we get the exact sequence
$$
E_0\stackrel{\phi}{\longrightarrow} E_1\to\Ob_X\to0.
$$
Consider the following exact sequence 
\begin{equation}\label{exact_resolution1}
E_0\stackrel{(\phi,0)}{\longrightarrow} E_1\oplus\oO_X\to\Ob_X\oplus\oO_X\to0.
\end{equation}
Let $\tau$ be the tautological section of $\pi_{E_1\oplus\oO_X}^*(E^{-1}\oplus\oO_X)$.  From \cite[Lemma 2.1]{JT}, the abelian cone 
\begin{multline}\label{abelian_cone1}
C(\Ob_X\oplus \oO_X) \text{~is cut out of~}
C(E_1\oplus \oO_X)=\mbox{Tot}(E^{-1}\oplus\oO_X) \text{~by the section~} \\
\pi^*_{E_1\oplus\oO_X}\phi^*(\tau) \text{~of~} \pi^*_{E_1\oplus\oO_X} E^0.
\end{multline}

From the construction of the projective bundle (cone),  this tautological section 
$\tau$
of $\pi_{E_1\oplus\oO_X}^*(E^{-1}\oplus\oO_X)$
 induces a homomorphism  from $\oO_{\mathbb{P}(E^{-1}\oplus \oO_X)}(-1)$ to 
$\tau$ on the projective bundle $\mathbb{P}(E^{-1}\oplus\oO_X)$, which we denote it by $\tau(-1)$.  Therefore
\begin{multline}\label{projective_cone1}
\overline{N} \text{~is cut out of~}
\overline{C}(E_1)=\mathbb{P}(E^{-1}\oplus\oO_X) \text{~by the section ~} 
\overline{\pi}^*_{E_1}(\phi,0)^*(\tau(1)) \text{~of~} 
\overline{\pi}_{E_1}^*E^0(1).
\end{multline}

\subsubsection{Local model} 
We work on the local model of $\overline{N}$. 
 We have 
 $$
\big\{T_A|_X\oplus \oO_X\stackrel{(ds,0)}{\longrightarrow}E|_X\oplus \oO_X\big\} \quad\text{is quasi-isomorphic to }\quad
\big\{E_0\oplus\oO_X\to E_1\oplus\O_X\big\}.
$$
Therefore the ideal of $\overline{N}$ in the smooth ambient space $P:=\mathbb{P}(E^*\oplus\oO_A)$ is
\begin{equation}\label{ideal_bar}
\big(\pi^*_{E}s,\,\overline{\pi}^*_{E}(ds,0)^*(\tau(1)) \big).
\end{equation}
Here we recall that 
$$\overline{\pi}^*_{E}: \mathbb{P}(E^*\oplus \oO_A)\to A$$ is the projection of the projective bundle.

We can take $A=\aaa_{\kappa}^n$ and the vector bundle $E=A\times \aaa_{\kappa}^r$ is trivial. We work in local coordinates $x_i$ for $A$ such that $(s_1,\cdots, s_r)$ are functions of $x_i$. Trivializing $E$ (with rank $r$) with a basis of sections $e_j$, we get a dual basis $f_j$ for $E^*$ and coordinates $y_j$ on the fibers of Tot$\;(E^*)$.
From the construction of $P:=\mathbb{P}(E^*\oplus \oO_A)=A\times \mathbb{P}^r$, and let $D_\infty=\mathbb{P}(E^*)=A\times \mathbb{P}^{r-1}$ be the infinity divisor. 
In the local coordinates above, one takes the homogeneous coordinates of of the fibre of $A\times \mathbb{P}^r$ by 
$$[u_1:\cdots : u_r : u_{r+1}].$$ Therefore $D_\infty=\{u_{r+1}=0\}\subset A\times \mathbb{P}^r$.

Recall that the   log cotangent bundle $\Omega_{P}^{\log}$ is defined as the sheaf of differential forms on $P$ with logarithmic poles along $D_\infty$, i.e.,  to be the sheaf of meromorphic $1$-forms on $P$ that are holomorphic away from $D_\infty$ and locally on $A\times U_i$ in coordinates $\{z, z_i=\frac{u_l}{u_i}\}$ along $D_\infty$ can be written as
$$
f\frac{dz}{z}+\sum_{i} f_i dz_i
$$
with all $f, f_i$  holomorphic functions.

Let us look at the meromorphic function $\widetilde s=\sum_{j=1}^r s_ju_j$ on $P$ which  only has linear terms on the homogeneous coordinates $\{u_i\}$ of $\mathbb{P}^r$.
The function $\widetilde s$ extends to $P$ and has only  first order pole along $D_\infty$, since in local coordinates $\widetilde s=\frac{g}{z}$ for some regular function $g$ where $z$ is the normal coordinate of $D_\infty$. 
Therefore $\widetilde s\in \oO_{P}(D_\infty)$.  
Then the differential $d\widetilde s$ has pole along $D_\infty$ and gives a section of the twisted  log cotangent bundle $\Omega_{P}^{\log}(D_\infty)$.

The differential  $d\widetilde s$ also gives the ideal 
$$(s_1,\cdots, s_r; \sum_iu_ids_i)$$
in coordinates $\{x_i\}$ of $A$ and homogeneous coordinates $\{u_i\}$ of $\mathbb{P}^r$
which is (\ref{ideal_bar}) for the compactification $\overline{N}$.

\subsubsection{Calculation on each affine charts}

We write down the ideal  (\ref{ideal_bar}) on each affine charts of $P$. 

$P$ is covered by $r+1$ affine open subset $A\times U_i$, where $U_i=\{u_i\neq 0\}\subset \mathbb{P}^r$. 
On each $U_i$, the coordinates are given by:
$$\{y_1^{i}:=\frac{u_1}{u_{i}}, \cdots, \hat{1}, \cdots,  y_r^i=\frac{u_r}{u_i},  y_i^{i}=\frac{u_r+1}{u_i}\}$$  
if $i\neq r+1$; and 
$$\{y_1^{r+1}:=\frac{u_1}{u_{r+1}}, \cdots,  y_r^{r+1}=\frac{u_r}{u_{r+1}}\}$$  
if $i=r+1$.

The scheme Tot$\;(E^*)=A\times U_{r+1}$ and  the fibre coordinates $\{y_1,\cdots, y_r\}$ of Tot$(E^*)$ are given by 
$\{\frac{u_1}{u_{r+1}}, \cdots, \frac{u_r}{u_{r+1}}\}$.  The regular function $\widetilde s=\sum_{i=1}^{r}y_i s_i$ and the critical locus of  $\widetilde s$ is $N$.

Let us look at the open affine $A\times U_i$ for $i\neq r+1$.  Let $z$ (think of $z$ as  $\frac{u_{r+1}}{u_i}$) be the local coordinate of $P$ in the normal direction of $D_\infty$ such that the zero locus of $z$ gives $D_\infty$.  The function
$$\widetilde s|_{A\times U_i}=s_i\frac{1}{z}+\sum_{\substack{l=1,\\
l\neq i}}^r s_l y_l^i\cdot \frac{1}{z}=\frac{1}{z}\cdot g$$
where $g= s_i+\sum_{\substack{l=1,\\
l\neq i}}^r s_l y_l^i$ is a regular function on $A\times U_i$.  Then 
$d\widetilde s=\frac{1}{z}\cdot dg-\frac{1}{z^2}\cdot g dz$. Thus 
$$d\widetilde s=\frac{1}{z}\left(ds_i+\sum_{\substack{l=1,\\
l\neq i}}^rds_l\cdot y_l^i + \sum_{\substack{l=1,\\
l\neq i}}^r s_l\cdot dy_l^i\right)-\frac{1}{z^2}\left( \sum_{\substack{l=1,\\
l\neq i}}^r s_l\cdot y_l^i +s_i\right)dz.$$

We use $\left(ds_i+ \sum_{\substack{l=1,\\
l\neq i}}^r ds_l\cdot y_l^i \right)$ to represent the ideal generated by 
$$\left( \frac{\partial s_i}{\partial x_1}+ \sum_{\substack{l=1,\\
l\neq i}}^r \frac{\partial s_l}{\partial x_1}\cdot y_l^i, \cdots,  \frac{\partial s_i}{\partial x_r}+ \sum_{\substack{l=1,\\
l\neq i}}^r \frac{\partial s_l}{\partial x_r}\cdot y_l^i\right).$$

\begin{lem}\label{lem_ideal_affine_chart}
The affine chart $\overline{N}\cap (A\times U_i)$ is given by the zero locus of the ideal 
\begin{equation}\label{ideal_affine_chart}
\left(s_1,\cdots, s_i, \cdots, s_r; \left(ds_i+\sum_{\substack{l=1,\\
l\neq i}}^r ds_l\cdot y_l^i \right)\right).
\end{equation}
\end{lem}
\begin{proof}
First in this local case the ideal (\ref{ideal_bar}) is given by: 
\begin{equation}\label{ideal_lem}
(s_1,\cdots, s_r; \sum_iu_ids_i).
\end{equation}
Here $\left(\sum_iu_ids_i \right)$ means the ideal generated by:
$$\left(\sum_iu_i\frac{\partial s_i}{\partial x_1}, \cdots, \sum_iu_i\frac{\partial s_i}{\partial x_r}\right).$$

To write down the ideal (\ref{ideal_lem}) in 
 this affine chart $\overline{N}\cap (A\times U_i)$, we first have 
 $$\sum_i u_i ds_i= \sum_i \frac{u_i}{u_{r+1}}ds_i$$ in the affine chart $\overline{N}\cap (A\times U_{r+1})$. From the relation between the coordinates on $U_i$ and $U_{r+1}$,
 $$\frac{u_i}{u_{r+1}}=\frac{1}{\frac{u_{r+1}}{u_i}}=\frac{1}{z};$$
 and for $l=1,\cdots, r$ but $l\neq i$, 
 $$ \frac{u_l}{u_{r+1}}=\frac{u_l}{u_i}\cdot \frac{u_{i}}{u_{r+1}}=\frac{u_l}{u_i}\cdot\frac{1}{z},$$
 then in the open affine chart  $\overline{N}\cap A\times U_i$, 
$\sum_i u_i ds_i$ can be written as 
$$ds_i+ \sum_{\substack{l=1\\
i\neq i}}^{r}\frac{u_l}{u_{i}}\cdot \frac{1}{z}ds_l.$$
Since in this affine chart 
$u_i\neq 0$, the ideal (\ref{ideal_lem}) in this open affine chart is given by 
$(s_1,\cdots, s_i, \cdots, s_r)$ and 
$$ds_i+ \sum_{\substack{l=1\\
i\neq i}}^{r}\frac{u_l}{u_{i}}\cdot ds_l=ds_i+ \sum_{\substack{l=1\\
i\neq i}}^{r}y_l^i\cdot ds_l.$$
\end{proof}

The locus  $\overline{N}\cap (A\times U_i)$ determined by the ideal (\ref{ideal_affine_chart}) is not given by the critical locus of $g$ since the ideal $dg$ is 
$$\left(s_1,\cdots, \widehat{s}_i, \cdots, s_r; \left(ds_i+\sum_{\substack{l=1,\\
l\neq i}}^r ds_l\cdot y_l^i \right)\right)$$
where $\widehat{s}_i$ means this $s_i$ is not included.  If the ideal $(s_i)\subseteq (ds_i)$, then the ideal (\ref{ideal_affine_chart}) is  given by the  ideal $dg$.

\subsubsection{$d$-critical scheme structure}
We aim to show that $\overline{N}$ is also a $d$-critical scheme.  We first consider $A\times U_{r+1}$. Then in this case let $y_l^{r+1}:=\frac{u_l}{u_{r+1}}$, then the function 
$\widetilde s=\sum_{l=1}^{r}s_l y_l$ and we go back to the case of  abelian cone $N$, which is the critical locus of 
$\widetilde s: \mbox{Tot}(E^*)A\times U_{r+1}\to \aaa_{\kappa}^1$.  Hence 
$$(\Crit(\widetilde s), A\times U_{r+1}, \widetilde s, i_{r+1})$$
is a critical chart, where $i_{r+1}: \Crit(\widetilde s)\hookrightarrow A\times U_{r+1}$ is the closed embedding. 

Let us consider the affine scheme $A\times U_i (i\neq r+1)$. From Lemma \ref{lem_ideal_affine_chart}, 
$$\widetilde s|_{A\times U_i}=\frac{s_i}{z}+\sum_{\substack{l=1,\\
 l\neq i}}^r s_l\frac{u_l}{u_i}\frac{1}{z}:=\frac{1}{z}\cdot g$$
 where $g=s_i+\sum_{\substack{l=1,\\
 l\neq i}}^r s_l y_l^i$.   
 Suppose that $(s_i)\subset (ds_i)$, 
  we have critical chart 
 $$(\Crit(g), A\times U_i, g, i_{i})$$
 where $i_i: \Crit(g)\hookrightarrow A\times U_i$ is the closed embedding.  Since the differential 
 $$dg=ds_i+ \sum_{\substack{l=1,\\
 l\neq i}}^r (ds_l y_l^i+ s_l d y_l^i),$$ one can write down the ideal of $\Crit(g)$:
 \begin{equation}\label{ideal_Ui}
 \left( \{s_l\}_{l=1, l\neq i}^{r}; ds_1\cdot y_1^i, \cdots, ds_i, \cdots,  ds_r\cdot y_r^i\right).
 \end{equation}

Next we consider another affine scheme $A\times U_j (j\neq i, r+1)$.  Let 
$$\{y_1^{j}:=\frac{u_1}{u_{j}}, \cdots, \hat{1}, \cdots,  y_r^j=\frac{u_r}{u_j},  y_j^{j}=\frac{u_r+1}{u_j}\}$$  
be the local coordinate of $U_j$.  Let $z=\frac{u_{r+1}}{u_j}$.  Then we have 
$$\widetilde s|_{A\times U_j}=\frac{s_j}{z}+\sum_{\substack{l=1,\\
 l\neq j}}^r s_l\frac{u_l}{u_j}\frac{1}{z}:=\frac{1}{z}\cdot g^\prime$$
 where $g^\prime=s_j+\sum_{\substack{l=1,\\
 l\neq j}}^r s_l y_l^j$.   Therefore we have critical chart 
 $$(\Crit(g^\prime), A\times U_j, g, i_{j})$$
 where $i_j: \Crit(g^\prime)\hookrightarrow A\times U_j$ is the closed embedding.  The differential 
 $$dg^\prime=ds_j+ \sum_{\substack{l=1,\\
 l\neq j}}^r (ds_l y_l^j+ s_l d y_l^j),$$ one can write down the ideal of $\Crit(g^\prime)$:
 \begin{equation}\label{ideal_Uj}
 \left( \{s_l\}_{l=1, l\neq j}^{r}; ds_1\cdot y_1^j, \cdots, ds_j, \cdots, ds_r\cdot y_r^j\right).
 \end{equation}
 
 We check that the critical locus of $\widetilde s$ and $g$ are the same on the intersection $A\times U_{r+1}\cap U_i$.  Let us write down for $\widetilde s: A\times U_{r+1}\to \aaa_{\kappa}^1$, 
 $$d\widetilde s=\sum_{l=1}^r (ds_l y_l^{r+1}+ s_l d y_l^{r+1}).$$
 Then the ideal of $\Crit(\widetilde s)$ is given by 
 $$(s_1, \cdots, s_r; ds_1 \cdot y_1^{r+1}, \cdots, ds_r \cdot y_r^{r+1}).$$
 Comparing with the ideal (\ref{ideal_Ui}),  note that $y_i^{r+1}=\frac{1}{y_i^i}\neq 0$, and the ideal 
 $(ds_i)\subset (s_i)$,  therefore they are the same.  The comparison of  the ideal (\ref{ideal_Ui}) and  the ideal (\ref{ideal_Uj}) is similar. 

Let $\sS_{\overline{N}}$ be the unique coherent sheaf on $\overline{N}$ constructed in Theorem \ref{property:local:data}. Then  for the critical charts above, we have the following exact sequences
of sheaves of $\kappa$-vector spaces:
$$0\rightarrow\sS_{\overline{N}}|_{\Crit{\widetilde s}}\stackrel{\iota}{\longrightarrow}
\frac{i_{r+1}^{-1}(\oO_{A\times U_{r+1}})}{I^{2}_{\Crit(\widetilde s),A\times U_{r+1}}}\stackrel{d}{\longrightarrow}
\frac{i_{r+1}^{-1}(T^*(A\times U_{r+1}))}{I_{\Crit(\widetilde s),A\times U_{r+1}}\cdot i_{r+1}^{-1}(T^*(A\times U_{r+1}))}$$
where $d$ maps $f+I_{\Crit(\widetilde s),A\times U_{r+1}}^2$ to 
$df+I_{\Crit(\widetilde s),A\times U_{r+1}}\cdot i_{r+1}^{-1}(T^*(A\times U_{r+1}))$;
and 
$$0\rightarrow\sS_{\overline{N}}|_{\Crit(g)}\stackrel{\iota}{\longrightarrow}
\frac{i_{i}^{-1}(\oO_{A\times U_{i}})}{I^{2}_{\Crit(g),A\times U_{i}}}\stackrel{d}{\longrightarrow}
\frac{i_i^{-1}(T^*(A\times U_{i}))}{I_{\Crit(g),A\times U_{i}}\cdot i_i^{-1}(T^*(A\times U_{i}))}$$
where $d$ maps $f+I_{\Crit(g),A\times U_{i}}^2$ to 
$df+I_{\Crit(g),A\times U_{i}}\cdot i_i^{-1}(T^*(A\times U_{i}))$.
Then we can take a section $s\in H^0(\sS^0_{\overline{N}})$ such that: 
for any $x\in \Crit(\widetilde s)$ in the open affine $A\times U_{r+1}$, 
$$\iota_{\Crit(\widetilde s), A\times U_{r+1}}(s|_{\Crit(\widetilde s)})=i_{r+1}^{-1}(\widetilde s)+I^2_{\Crit(g),A\times U_{i}};$$ 
And or any $x\in \Crit(g)$ in the open affine $A\times U_{i}$ for $i\neq r+1$, 
$$\iota_{\Crit(g), A\times U_{i}}(s|_{\Crit(g)})=i_{i}^{-1}(g)+I^2_{\Crit(g),A\times U_{i}}.$$ 
These local sections $s|_{\Crit(\widetilde s)}$ and $s|_{\Crit(g)}$ glue since $\widetilde s$ and $g$ are the same at the intersection $A\times (U_{r+1}\cap U_i)$. Thus we show that: 

\begin{thm}\label{thm_d_critical_scheme_BarN}
If on each affine chart, the sections $s_i$ satisfies the condition 
$(s_i)\subset (ds_i)$, then 
the pair $(\overline{N}, s)$ is a $d$-critical scheme. 
\end{thm}

Then from Theorem \ref{thm_d_critical_scheme_symmetric_semi_obstruction_theory}, $\overline{N}$ admits a symmetric semi-perfect obstruction theory $\phi$, therefore a virtual fundamental cycle $[\overline{N}, \phi]^{\virt}$
of degree zero. The invariant is 
\begin{equation}\label{invaraint_compactification_semi-SPOT}
\int_{[\overline{N},\phi]^{\virt}}1.
\end{equation}

\begin{example}\label{example_d_critical_scheme_cone_2}
We change notation here and let $N=\spec\kappa[x, y]/(x^2, xy)$, and $N$ is the critical locus of $f=x^2y$.
Also $N$ is the abelian cone over $X=\spec \kappa[x]/(x^2)$. 
The scheme $\overline{N}=\proj R[y_0:y_1]/(x^2, xy_0)$ in Example \ref{example_d_critical_scheme_cone}
is the compactification of $N$, and is also a $d$-critical scheme. 

We calculate the invariant.  There is a $\G_m$ action on $\overline{N}$ induced from the $\G_m$-action on $\aaa_{\kappa}^1\times \mathbb{P}^1$ where the action on $\aaa_{\kappa}^1$ trivial. 
There are two torus fixed points $0$ and $\infty$, where $0$ is a fat point with multiplicity $2$, and $\infty$ is a smooth point.  One can calculate the Behrend function on them.  Example 3.1 in \cite{JT} calculated $\nu_{\overline{N}}(0)=1$, since the Behrend function depends only on its local neighborhood hence it is the same as the Behrend function of $\nu_N$ at $0$.  Since $\infty$ is a smooth point, $\nu_{\overline{N}}(\infty)=1$.  
So $\chi(\overline{N}, \nu_{\overline{N}})=2=\int_{[\overline{N},\phi]^{\virt}}1$.

The compactification $\overline{N}$ can also be taken as the non-reduced projective subscheme in 
$\mathbb{P}^2=\proj(\kappa[x:y_1:y_2])$ by the homogeneous ideal
$$(x^2, xy_1).$$
Then we can take a deformation family of $\overline{N}$ by considering the subschemes given by the ideals 
$$(x^2-t^2y_2^2, xy_1).$$
When $t=0$, this is the same as $\overline{N}$. 
When $t\neq 0$, this gives the subscheme determined by the ideal 
$(x^2-t^2y_2^2)$ in $\mathbb{P}^1=\proj(\kappa[x:0:y_2])$, hence it is the subscheme containing two smooth points.  This shows that the invariant $\int_{[\overline{N},\phi]^{\virt}}1$  is deformation invariant. 
\end{example}

\subsection{Application to the $\G_m$-equivariant symmetric semi-perfect obstruction theory}

We generalize the result in \cite{Jiang2} to symmetric  semi-perfect obstruction theory. We first recall the $\G_m$-equivariant semi-perfect obstruction theory.  

\begin{defn}(\cite{Kiem})
A $\G_m$-equivariant symmetric semi-perfect obstruction theory on $X$ consists of the following:
\begin{enumerate}
\item a $\G_m$-equivariant \'etale open cover 
$$\{X_\alpha\to X\}$$
of $X$;
\item the symmetric obstruction theory 
$$\phi_\alpha: E_\alpha^\bullet\to L_{X_\alpha}^\bullet$$
is $\G_m$-equivariant in the equivariant derived category $D^b(X_\alpha)$ of coherent sheaves. 
\end{enumerate}
\end{defn}

\begin{rmk}
Note that in \cite{BF2}, a $\G_m$-symmetric equivariant  obstruction theory is given by a $\G_m$-equivariant morphism 
$$\phi_\alpha: E_\alpha^\bullet\to L_{X_\alpha}^\bullet$$
such that there exists a $\G_m$-equivariant bilinear form 
$$\Theta: E^\bullet_\alpha\stackrel{\sim}{\longrightarrow} E_\alpha^{\bullet\vee}[-1]$$
in $D^b(X_\alpha)$ such that $\Theta$ is an isomorphism satisfying $\Theta^\vee[1]=\Theta$.  A $\G_m$-symmetric equivariant obstruction theory requires more than just a $\G_m$-equivariant obstruction theory, see 
\cite[\S 3.4]{BF2}. 
\end{rmk}

The main result in \cite[Theorem 5.20]{Jiang2} does not require that $X$ admits a ``symmetric" equivariant obstruction theory, but rather a $\G_m$-equivariant perfect obstruction theory.  First we have the following result, which is the same as Theorem 5.8 in \cite{Jiang2}. 
\begin{prop}\label{prop_fixed_locus_Behrend_function}
Let $X$ be a scheme or a Deligne-Mumford stack which admits a symmetric semi-perfect obstruction theory. Assume that there exists a $\G_m$-action on $X$ with proper fixed locus $F\subseteq X$ such that 
$$\{E_\alpha^\bullet\to L_{X_\alpha}^\bullet\}_{\alpha\in\Lambda}$$
is a $\G_m$-equivariant obstruction theory.  Then 
$$\chi(X, \nu_X)=\chi(F, \nu_{X}|_{F}).$$
Moreover, if $X$ is proper, 
$$\int_{[X, \phi]^{\virt}}1=\chi(F, \nu_{X}|_{F}).$$
\end{prop}
\begin{proof}
The first result is from the fact that the Behrend function $\nu_X$ is constant on the nontrivial $\G_m$-orbits. 
If $X$ is proper, the last result is from Theorem \ref{thm_Behrend_Euler_semiPOT}. 
\end{proof}

\subsubsection*{Cosection localization}

It is important to let $X$ non-proper. Still the $\G_m$-action on $X$ defines a cosection
$$\sigma: \Omega_X\to \oO_X$$
by taking the dual of the canonical vector field $v_x\mapsto \frac{d}{d\lambda}\mu(\lambda\cdot x)|_{\lambda=1}$ on $X$ by the $\G_m$-action. The degenerate locus of $\sigma$ is $D(\sigma)=F$.  We assume that $X$ is quasi-projective if it is a Deligne-Mumford stack, therefore there exists a closed immersion 
$$X\hookrightarrow M$$
into a smooth higher dimensional smooth Deligne-Mumford stack $M$.  Hence we have the following cartesian diagram:
\[
\xymatrix{
C\ar@{^{(}->}[r]\ar[d]& \Omega_{M}|_{X}\ar[d]\\
cv\ar@{^{(}->}[r]& \Omega_X
}
\]
similar to (\ref{diagram__N_Cone}), where $C\subset \Omega_M$ is the unique Lagrangian cone in $\Omega_M$ making the diagram commute.  Also \cite{KL} proves that 
$$C\subset  \Omega_{M}|_{F}\sqcup\ker( \Omega_{M}|_{X\setminus F}\to \oO_X).$$
Then Kiem-Li define the cosection localized virtual cycle by applying the localized Gysin map in \cite[\S 3.3]{KL}. We use the analytic method in the  Appendix of \cite{KL}, and 
take a small perturbation $\xi$ is the zero section $M$ of $\Omega_{M}$ such that $\xi\cap C$ only supports on $F$.  Then  
$$[X]^{\virt}_{\loc}:=\xi\cap C \in A_0(F).$$
\begin{thm}
Let $F$ be proper. 
We have:
$$\int_{[X]^{\virt}_{\loc}}1=\chi(F, \nu_X|_{F}).$$
\end{thm}
\begin{proof}
We only need to show that 
$$\chi(X, \nu_X)=\#(\xi\cap C),$$
the intersection number, but simialr to the 
proof in \cite[\S 4]{JT} or \cite[\S 5.8]{Jiang2}, this is the global index theorem due to Theorem 9.7.11 of Kashiwara-Schpaira \cite{KS} since the characteristic cycle of $\nu_X$ is $C$.  In the Deligne-Mumford stack case, the index theorem is due to Maulik-Treumann in \cite{MT}. 
\end{proof}


\subsection*{}

\end{document}